\documentclass[11pt]{amsart}
\usepackage[margin=1.15in]{geometry}                
\usepackage{graphicx}
\usepackage{amssymb,amscd,amsthm}
\usepackage{epstopdf}
\usepackage{tikz}
\usepackage[colorlinks=true, pdfstartview=FitV, linkcolor=blue, citecolor=blue, urlcolor=blue]{hyperref}
\usepackage[utf8]{inputenc}

\newcommand{\fm}{{\mathfrak m}}

\newcommand{\C}{{\mathbb C}}
\newcommand{\N}{{\mathbb N}}
\renewcommand{\P}{{\mathbb P}}

\newcommand{\A}{{\mathcal A}}
\renewcommand{\L}{{\mathcal L}}

\def\Ass{\operatorname{Ass}}
\DeclareMathOperator{\codim}{codim}
\DeclareMathOperator{\ess}{ess}
\def\Jac{\operatorname{Jac}}

\DeclareMathOperator{\depth}{depth}
\def\height{\operatorname{ht}}

\def\Spec{\operatorname{Spec}}

\DeclareMathOperator{\rank}{rank}
\DeclareMathOperator{\reg}{reg}
\DeclareMathOperator{\het}{ht}
\DeclareMathOperator{\Sym}{Sym}

\DeclareMathOperator{\GL}{GL}
\DeclareMathOperator{\PGL}{PGL}

\DeclareMathOperator{\Der}{Der}

\theoremstyle{plain} 
\newtheorem{thmA}{Theorem}
\newtheorem{thm}{Theorem}[section]

\newtheorem*{introthm*}{Theorem}

\newtheorem{cor}[thm]{Corollary}
\newtheorem{lem}[thm]{Lemma}
\newtheorem{prop}[thm]{Proposition}

\theoremstyle{definition}
\newtheorem{defn}[thm]{Definition}

\newtheorem{qst}[thm]{Question}

\theoremstyle{remark}
\newtheorem{rem}[thm]{Remark}

\title[Containments for singular loci of reflection arrangements]{ Singular loci of reflection arrangements\\ and  the containment problem}
\author{Benjamin Drabkin}
\address{Department of Mathematics, University of Nebraska-Lincoln, 203 Avery Hall, Lincoln, NE 68588}
\email{benjamin.drabkin@huskers.unl.edu}
\author{Alexandra Seceleanu }
\address{Department of Mathematics, University of Nebraska-Lincoln, 203 Avery Hall, Lincoln, NE 68588}
\email{aseceleanu@unl.edu}

\date{Revised 01/05/2021}                                           

\begin{document}

\begin{abstract}
This paper provides insights into the role of symmetry in studying polynomial functions vanishing to high order on an algebraic variety. The varieties we study are singular loci of hyperplane arrangements in projective space, with emphasis on arrangements arising from complex reflection groups. We provide minimal sets of equations for the radical ideals defining these singular loci and study containments between the ordinary and symbolic powers of these ideals. Our work ties together and generalizes results in \cite{BDH, DST, HS, MS1} under a unified approach.
\end{abstract}

\maketitle

\vspace{-0.5em}
\section{Introduction}

The objective of this paper is to provide insights into the role of symmetry in studying polynomial functions vanishing to high order on an algebraic variety. 

To formalize the latter concept, recall that for an integer $r\geq 0$, the $r$-th symbolic power of a radical ideal $I$ is defined to be \[I^{(r)}=\bigcap_{P\in\Ass(R/I)}(I^rR_P\cap R).\]  Symbolic powers of ideals are interesting for a number of reasons not least of which is that, for a radical ideal $I\subseteq R=\C[x_0,\ldots,x_n]$ the $r$-th symbolic power $I^{(r)}$ is the ideal of all polynomials vanishing to order at least $r$ on the variety defined by $I$ according to the Zariski-Nagata theorem.

We bring an influx of symmetry into the study of symbolic powers by considering the case of ideals $I$ which arise from the action of a complex reflection group. To be precise, any finite group $G$ generated by pseudoreflections determines an arrangement $\A=\A(G)\subseteq \C^{\rank(G)}$ of hyperplanes, each of which are fixed pointwise by one of the reflections in $G$. We focus our study on symbolic powers of radical ideals $J(\A)$ defining the singular locus of reflection arrangements $\A$. All of these ideals are equidimensional of codimension two.

Our interest in singular loci of hyperplane arrangements has been sparked by the peculiar behavior of some ideals in this class with regards to containments between ordinary and symbolic powers.  It is known thanks to  \cite{ELS, HH, MS}  that the containments $J(\A)^{(2r)}\subseteq J(\A)^r$ are satisfied for every positive integer $r$. What is more interesting, however, is that several examples of ideals $J(\A)$ have arisen in the literature as witnesses to the optimality of the above containment, in the sense that they have also been shown to satisfy $J(\A)^{(3)}\not \subseteq J(\A)^2$ for certain groups $G$. In hindsight, the groups for which the stated noncontainment was known to hold before our work are the infinite family of monomial groups $G(m,m,3)$ \cite{DST, HS} and two classical groups studied by Klein ($G_{24}$) and Wiman ($G_{27}$)  \cite{BNAL, BDH}.

In this paper we complete the picture that the previously referenced works have started to outline by taking up the challenge of classifying which singular loci of reflection arrangements satisfy the containment $J(\A)^{(3)}\subseteq J(\A)^2$ and which do not. In the reflection arrangement literature the classification of  the irreducible complex pseudoreflection groups  by Shephard and Todd \cite{ST} in terms of an infinite family $G(m,p,n)$ and 33 sporadic groups denoted $G_4 -G_{37}$ is fundamental. We express our results in terms of their classification:
\begin{thmA}[Theorem \ref{thm:Areprise}]
\label{Thm:A}
Let $G$ be a finite complex reflection group with reflection arrangement $\A$.  Then $J(\mathcal{A})^{(3)}\subseteq J(\mathcal{A})^2$ if and only if no irreducible factor of $G$ is isomorphic to one of the following groups 
\[
G_{24},G_{27},G_{29},G_{33},G_{34}, \text{ or } G(m,m,n) \text{ with }m,n\geq 3.
\]
\end{thmA}

Although several examples, including those given by the theorem above, show the optimality of the containment  $J(\A)^{(2r)}\subseteq J(\A)^r$ for $r=2$, a tighter containment conjectured by Harbourne has been shown to hold for many interesting classes of ideals. In the context of codimension two ideals, such as $J(\A)$, Harbourne's proposed containment is $J(\A)^{(2r-1)}\subseteq J(\A)^r$. We provide some new evidence for a stable version of Harbourne's containment introduced in \cite{Gr}, $J(\A)^{(2r-1)}\subseteq J(\A)^r$ for $r\gg 0$, by showing that this improved containment is valid for singular loci of products of reflection arrangements in $\P^2$ for all integers $r\geq 3$.

\begin{thmA}[Theorem \ref{thm:Breprise}]
\label{thmB}
Let $G$ be a finite complex reflection group with irreducible factors of rank three and corresponding reflection arrangement $\A$.  Then for all integers $r\geq 3$ the containment $J(\mathcal{A})^{(2r-1)}\subseteq J(\mathcal{A})^r$ holds.
\end{thmA}

Our methods for analyzing (non)containments rely heavily on the structure of the ideals $J(\mathcal{A})$ and their presentations, which we find to be particularly interesting its own right. In algebraic language we can summarize our findings as follows: the ideal $J(\mathcal{A})$ admits a linear relation (syzygy) among its minimal generators if and only if the containment $J(\mathcal{A})^{(3)}\subseteq J(\mathcal{A})^2$ holds; see Corollary \ref{cor:linearsyz}. The groups for which there is no such linear relation are precisely the ones singled out for being exceptions to this containment in the theorem above. Theorem \ref{thmB} shows that this distinction no longer persists in regards to containments between higher powers.

 In section \ref{s:equations} we give a complete description of  the defining equations for the reduced singular loci of complex reflection arrangements. This builds on ingredients which are fundamental in studying group actions, namely invariant polynomials for the action of the reflection groups under consideration. The Chevalley-Shaphard-Todd theorem \cite{Chevalley, ST} characterizes reflection groups as those groups having polynomial rings of invariants with generators termed basic invariants. A modern counterpart to the study of polynomial invariants for group actions is the study of $G$-invariant derivations on the polynomial ring. For $G$ an irreducible reflection group these form a free module with basis elements referred to as basic derivations.
We show the following relation between the invariants, derivations, and the singular locus:

\begin{thmA}[Theorem \ref{thm:eqJ}]
\label{Thm:B}
For an irreducible complex reflection group $G$ the ideal $J(\A)$ is minimally generated by the maximal minors of either the jacobian matrix for a set of $\rank(G)-1$ basic invariants of lowest degrees or by the coefficient matrix \eqref{eq:coeffmatrix} for a set of $\rank(G)-1$ basic derivations of lowest degrees.
\end{thmA}

To our knowledge this result is new and constitutes an improvement on a theorem of Steinberg \cite{Steinberg}, which gives set-theoretic determinantal equations for the loci of intersection of $r$ hyperplanes in $\A$ for each $1\leq r\leq \rank(G)$ in terms of the jacobian matrix of the basic invariants, as well as an improvement on \cite[Theorem 6.116]{OT}, which gives equations defining the singular locus of $\A$ set-theoretically (up to radical) in terms of minors of a coefficeint matrix of basic derivations.  
 We find it interesting to note that, as a consequence of our results, the minimal number of generators for ideals defining singular loci of irreducible reflection arrangements is always equal to the rank of the respective group. In particular, in embedding dimension three, i.e.~in $\P^2$, these ideals are almost complete intersections, meaning that they are generated by one more generator than their codimension. 

The structure of the paper is as follows. In section \ref{s:background} we introduce the main players of our paper, both from the world of hyperplane arrangements and that of containments between ordinary and symbolic powers. In section  \ref{s:equations} we establish the structure of the ideals defining the singular loci of reflection arrangements. The methods involved in establishing Theorem \ref{Thm:A} rely on  reducing the containment problem to checking it locally on lower-dimensional arrangements appropriately dubbed localizations of $\A$. Thus the backbone of the argument is given by an induction on $\rank(G)$, which we develop in section  \ref{s:localization}.
In the base cases when the containment in Theorem \ref{Thm:A} occurs, this can be read directly off the presentation (Hilbert-Burch) matrix for $J(\A)$ using the homological criteria of \cite{Se,Gr,GHM}. This provides new evidence for the usefulness of the explicit descriptions for the minimal generators and relation matrices for the ideals $J(\A)$ obtained in section \ref{s:equations}. Our results on (non) containments of the form $J(\mathcal{A})^{(2r-1)}\subseteq J(\mathcal{A})^r$ for $r=2$ are deduced in section \ref{s:containment} and for $r\geq 3$ in section \ref{s:questions}. This work opens up an array of questions which we also formulate in section \ref{s:questions}.

\section{Background}
\label{s:background}

\subsection{Reflection arrangements and their singular loci}
Let $\A$ be an arrangement of hyperplanes in the complex projective space $\P^n$ and denote the coordinate ring for the projective space $R=\C[x_0,\ldots,x_n]$. Denoting the equation of a hyperplane $H$  by $\ell_H$, the ideal defining the arrangement is the principal ideal $(F_\A)$, where $F_\A=\prod_{H\in A}\ell_H$. 

In this paper we focus on the ideals defining the singular loci of arrangements of hyperplanes. The singular locus of $\A$ is  the vanishing locus of the jacobian ideal of $F_\A$,  namely 
$\Jac(F_\A)=\left (\frac{\partial F_\A}{\partial x_i} , 0\leq i\leq n\right)$. While this jacobian ideal typically gives a nonreduced scheme structure on the singular locus of $\A$, throughout this paper we are concerned with the radical ideal defining the reduced singular locus of $F_\A$, namely  $J(\A)=\sqrt{\Jac(F_\A)}$.

One of the main class of examples of hyperplane arrangements is given by reflection arrangements. A {\em pseudoreflection} is a linear transformation different from the identity that fixes a hyperplane pointwise and has finite order (not necessarily two) as an element of $\GL_{n+1}(\C)$. A hyperplane  arrangement $\A$ is called a {\em reflection arrangement} if there is a finite group $G$ generated by pseudoreflections such that the hyperplanes of $\A$ are the hyperplanes pointwise fixed by the elements of $G$ that are pseudoreflections. Note that the hyperplane fixed by a pseudoreflection is uniquely determined by the class of the pseudoreflection in $\PGL_n(\C)$ and thus it suffices to consider unitary pseudoreflections, that is, we restrict to $G\subseteq \PGL_n(\C)$. A finite subgroup $G\subseteq \PGL_n(\C)$ generated by pseudoreflections is termed a  {\em pseudoreflection group} and its reflection arrangement is denoted $\A(G)$. 

Pseudoreflection groups are characterized by the fact that their rings of invariants are regular \cite{Chevalley, ST}.  More precisely, $G$ is a pseudoreflection group if and only if $R^G=\C[f_0,\ldots,f_n]$, where the polynomials $f_0,\ldots, f_n$, called the {\em basic invariants} of $G$, are algebraically independent. While the basic invariant polynomials are not unique, their degrees are uniquely determined by $G$ and we adopt the convention that $\deg(f_0)\leq \deg(f_1)\leq \ldots \leq \deg(f_n)$. The integers $\deg(f_i)-1$ are known as the {\em exponents} of $G$. The basic invariants are closely related to the defining equation of the arrangement $\A$. Specifically, denoting the jacobian matrix of the basic invariants by  
\begin{equation}
\label{eq:jacobian}
 \Jac(f_0,\ldots,f_n)=\begin{bmatrix} \frac{\partial f_j}{\partial x_i}\end{bmatrix}_{0\leq i,j\leq n}
 \end{equation}
 and the order of the reflection fixing the hyperplane $H$  by $e_H$, one has by \cite{Steinberg} that 
 \[
 \det \Jac(f_0,\ldots,f_n) =\prod_{H\in\A} (\ell_H)^{e_H-1}
 \]
 and in particular 
 \[
  (F_\A)=\sqrt{\left( \Jac(f_0,\ldots,f_n)\right)}.
 \]
  Note that our convention is to list the partial derivatives of each invariant polynomial as a column of the Jacobian matrix.

 The \textit{module of }$\C$\textit{-derivations} on $R$ is $\Der_\C(R)=\bigoplus_{i=0}^n \frac{\partial}{\partial x_i}R$. The action of the group $G$ on $R$ induces an action on $\Der_\C(R)$ given by $(g\theta)(r)=g(\theta(g^{-1}r))$ for  $g\in G, \theta\in \Der_\C(R)$ and $r\in R$. An important feature of pseudoreflection groups is that the modules of $G$-\textit{invariant derivations} $\Der_R^G$ are free $R$-modules \cite[Lemma 6.48]{OT}. We shall refer to a basis of homogeneous elements  $\{\theta_0,\ldots, \theta_n\}$ for $\Der_R^G$ as a set of {\em basic derivations}. As in the case of the basic invariants, only the degrees of the basic derivations are uniquely determined, not the basic derivations themselves. The integers $\deg(\theta_i)-1$ are referred to as {\em coexponents} for the group $G$.  Each basic derivation can be written in terms of the basis for $\Der_\C(R)$ as 
\begin{equation}
\label{eq:coeffmatrix}
\theta_j=\sum_{i=0}^n d_{ij} \frac{\partial}{\partial x_i}, \text{ where } d_{ij}=\theta_j(x_i),
\end{equation}
which gives rise to the {\em coefficient matrix} $Q(\theta_0,\ldots, \theta_n)=\begin{bmatrix} d_{ij} \end{bmatrix}_{0,\leq i,j\leq n}$. The coefficient matrix is even more closely related to the defining equation of the reflection arrangement $\A$ than the jacobian matrix by the identity
 \[
 \det \left(Q(\theta_0,\ldots, \theta_n)\right) = \prod_{H\in\A} \ell_H =F_\A.
 \]
Comparing this to the identity regarding the jacobian determinant displayed above gives the intuition that the jacobian matrix describes the hyperplane arrangement up to radical, while the coefficient matrix takes it one step further describing its reduced structure. In section \ref{s:equations} we give a description of  the defining equations for the reduced singular loci of complex reflection arrangements, which is reminiscent of the above formula. 

 Our work relies on the classification of  the irreducible complex pseudoreflection groups  by Shephard and Todd \cite{ST}. A pseudoreflection group $G\subseteq\PGL_n(\C)$ is called {\em irreducible} if there are no nontrivial subspaces $U,V$ closed under the action of $G$ such that $\C^{n+1}=U\oplus V$. The irreducible complex reflection  groups belong to an  infinite family $G(m,p,n+1)$ depending on 3 positive integer parameters with $p\mid m$, and 34 exceptional cases denoted $G_4$ through $G_{37}$.

Let  $\L(\A)$ be the set of all nonempty intersections of hyperplanes in $\A$, including $\P^n$ itself as the intersection over the empty set.  We call $\L(\A)$ the intersection lattice of $\A$ and any element of $\L(\A)$ is called a {\em flat} of $\A$. It is natural to think of $\L(\A)$ as a ranked lattice where the rank of a flat is its codimension. This results in a stratification of $\A$ by means of subvarieties consisting of the flats in $\L(\A)$ of codimension at most $c$ for each positive integer $c$. We explain in section \ref{s:equations} how, for an irreducible complex reflection group $G$, the components of this stratification correspond to rank conditions on $\Jac(f_0,\ldots,f_n)$ and $Q(\theta_0,\ldots, \theta_n)$. Furthermore, in section \ref{s:localization} we relate the associated primes for  $J(\A)^2$ to the defining ideals of certain flats in $\L(A)$.

\subsection{Containments between ordinary and symbolic powers}

Containment relationships between symbolic and ordinary powers are a source of great interest sparked by the proof in \cite{Sw} of a linear equivalence between the $I$-adic and symbolic toplogies. As an immediate consequence of the definition, $I^r\subseteq I^{(r)}$ for all $r$.  However, the other type of containment, namely that of a symbolic power in an ordinary power is much harder to pin down.  It has been proved by Ein-Lazarsfeld-Smith \cite{ELS}, Hochster-Huneke \cite{HH} and Ma-Schwede \cite{MS} that in a regular ring $R$ the containment $I^{(m)}\subseteq I^r$ holds for all $m\geq (dim(R)-1)r$, leaving open the question as to the extent to which this result is sharp.

A potential improvement was conjectured by Harbourne in \cite[Conjecture 8.4.3]{PSC}, and  in  \cite[Conjecture 4.1.1]{HaHu} in the case $e=n$, that $I^{(m)}\subseteq I^r$ for all $m\geq er-(e-1)$, where $e$ is the codimension of $V(I)$.  While this conjecture holds in a number of important cases, some counterexamples have been found.  Notably, most known counterexamples come from singular points of line arrangements:  one family of counterexamples known in the  literature under the name of Fermat configurations of points  \cite{DST, HS}, corresponds in hindsight to the singular loci of the monomial groups $G(m,m,3)$, while two other sporadic counterexamples known as the Klein and the Wiman configurations \cite{BDH} correspond to the singular loci of the groups $G_{24}$ and $G_{27}$ in the Shephard-Todd classification. The former family has been recently generalized to Fermat-like configurations of lines in $\mathbb P^3$ in \cite{MS1, MS2}, which correspond to the singular loci of rank four monomial groups $G(m,m,4)$. For each of the ideals $J$ defining one of these special configurations the non-containment $J^{(3)}\not\subseteq J^2$ has been proven in the cited source. 

The above-mentioned examples show the sharpness of the results in  \cite{ELS,HH,MS} for the pair $m=3, r=2$, leaving open this problem for all other pairs as well as Harbourne's conjecture for $r>2$. Moreover, while the papers \cite{MS1, MS2} give a negative answer to Harbourne's question in projective spaces of dimension $n>2$ along the lines of the Fermat examples in the plane, they leave open the possibility of higher dimensional counterexamples of sporadic flavor which would parallel the Klein and Wiman examples. Indeed, in this paper we find several new sporadic examples of hyperplane arrangements $\A$ one each in $\P^3,\P^4$ and $\P^5$ for which $J(\A)^{(3)}\not\subseteq J(\A)^2$. We also extend the results pertaining to the family of monomial groups to arbitrary rank.
Finally, we address the stable version of Harbourne's conjecture introduced in \cite{Gr}, proving that it is correct for singular loci of reflection arrangements of rank three and for $r\geq 3$.

\section{Defining equations}
\label{s:equations}

In this section the defining equations for the reduced singular loci of irreducible complex reflection arrangements are given. The following is our main result, which will be proven  by appealing to the Shephard-Todd classification.

\begin{thm}
\label{thm:eqJ}
Let $J(\A)$ be the homogeneous ideal defining the reduced singular locus of the reflection arrangement $\A$ corresponding to  an irreducible complex pseudoreflection group $G\subseteq \PGL_n(\C)$. Then the following hold:
\begin{enumerate}
\item $J(\A)$ is a perfect ideal of height 2,
\item the minimal number of generators of $J(\A)$ is equal to the rank of $G$, 
\item If  $G\not \in \{A_n,G_{25},G_{26}, G_{31}, G_{32}\}$, then $J(\A)$ is generated by the $n\times n$ minors of  the full Jacobian matrix 
\[
\Jac(f_0,\ldots,f_{n})=\begin{bmatrix} \frac{\partial f_j}{\partial x_i}\end{bmatrix}_{0\leq i, j\leq n},
\] 
and minimally generated by the $n\times n$ minors of its  submatrix
\[
\Jac(f_0,\ldots,f_{n-1})=\begin{bmatrix} \frac{\partial f_j}{\partial x_i}\end{bmatrix}_{0\leq i\leq n, 0\leq j\leq n-1},
\] where $f_0,\ldots,f_{n-1}$ are any $n$ basic invariants for $G$ of lowest degrees in a set of generators for $R^G$. In particular, ideals generated by these two sets of minors are equal and they are both radical.
\item  If $G\in\{G_{23}, G_{24}, G_{25}, G_{26}, G_{27}, G_{28}, G_{30}, G_{31}, G_{32},G_{35}, G_{36}\}$, then  $J(\A)$ is generated by the $n\times n$ minors of the coefficient matrix \eqref{eq:coeffmatrix} of a set of basic derivations 
\[
Q(\theta_0,\ldots, \theta_n)=\begin{bmatrix} \theta_j(x_i) \end{bmatrix}_{0\leq i, j\leq n},
\]
and is minimally generated by the $n\times n$ minors of a coefficient matrix for any $n$ elements of lowest degree in 
a set of basic derivations
\[
Q(\theta_0,\ldots, \theta_{n-1})=\begin{bmatrix} \theta_j(x_i) \end{bmatrix}_{0\leq i\leq n, 0\leq j\leq n-1}.
\]
In particular, the ideals generated by these two sets of minors are equal and they are both radical.
\end{enumerate}
\end{thm}

\begin{rem}
Parts (3) and (4) of Theorem \ref{thm:eqJ} above can be viewed as a generalization of the identities
\[
 \det \left(Q(\theta_0,\ldots, \theta_n)\right) = \prod_{H\in\A} \ell_H =F_\A 
 \quad \text{ and } \quad
 \det \Jac(f_0,\ldots,f_n) =\prod_{H\in\A} \ell_H^{e_H-1}.
 \]
Similarly, often the maximal minors of the submatrix $Q(\theta_0,\ldots,\theta_{n-1})$ cut out the singular locus of $\A(G)$ scheme-theoretically while the maximal minors of  $\Jac(f_0,\ldots,f_{n-1})$ define the same singular locus set theoretically. We emphasize that one cannot expect  the maximal minors of  $\Jac(f_0,\ldots,f_{n-1})$ to always define the singular locus of $\A$ ideal-theoretically.  Indeed a similar expression to the determinantal identity for $\Jac(f_0,\ldots,f_{n-1})$  can be obtained for lower order minors of the jacobian matrix of basic $G$-invariants. This shows that when the order of the reflection with fixed hyperplane $H$ is $e_H>2$ for some $H\in\A$ the respective jacobian minors are not square-free. Hence the ideal of submaximal minors of the jacobian matrix  cannot be expected to be radical when reflections of order greater than two are present. However part (3) of the theorem shows  that in the absence of reflections of order greater than two the ideal of submaximal minors of the jacobian matrix is indeed equal to $J(\A)$, with the notable exception of $\A=\A(G_{31})$.
\end{rem}

The remainder of the section is dedicated to the proof of the above theorem. From the definition of the singular locus it is clear that $J(\A)=\bigcap\limits_{1\leq i<j\leq t} (\ell_i,\ell_j)$ is an unmixed ideal of height two. Both statements claimed above follow from the Hilbert-Burch theorem once it is established that $J(\A)$ is the ideal of maximal minors of an $n\times(n+1)$ matrix. 


\subsection{General strategy}
To explain the relationship between the singular locus and the basic invariants of $G$ we begin with a classical result due to Steinberg.

\begin{lem}[Steinberg's theorem {\cite{Steinberg}}]
\label{lem:Steinberg}
Let $N=\Jac(f_0,\ldots,f_n)=\left[\frac{\partial f_j}{\partial x_i}\right]_{0\leq i, j\leq n}$ be the jacobian matrix of a set of basic invariants of a pseudoreflection group $G$ and let $p\in \P^n$ be any point. The following numbers are equal:
\begin{enumerate}
\item the nullity of $N$ at $p$
\item the maximum number of linearly independent hyperplanes of $\A$ passing through $p$.
\end{enumerate}
\end{lem}

There is also a counterpart of Steinberg's result for coefficient matrices of derivations.

\begin{lem}[{\cite[Theorem 6.113]{OT}}]
\label{lem:Steinbergcoeff}
Let $Q=Q(\theta_0,\ldots,\theta_n)=\begin{bmatrix} \theta_j(x_i) \end{bmatrix}_{1\leq i,j \leq n}$ be the coefficient matrix for a basis of the module of $G$-invariant derivations $\Der_R^G$ for a pseudoreflection group $G$ and let $p\in \P^n$ be any point. The following numbers are equal:
\begin{enumerate}
\item the nullity of $Q$ at $p$
\item the maximum number of linearly independent hyperplanes of $\A$ passing through $p$.
\end{enumerate}
\end{lem}

The previous results suffice to establish one containment of the identities in parts (3) and (4) of Theorem \ref{thm:eqJ}. 

\begin{cor}
\label{cor:InMsubsetJ}
Let  $f_0,\ldots f_{n}\in R^G$ be a set of basic invariants of an irreducible reflection group $G\subseteq \PGL_{n}(\C)$ with jacobian matrix $N=\Jac(f_0,\ldots,f_{n})$ and let $M$ be the $(n+1) \times n$ submatrix $M=\Jac(f_{i_0},\ldots,f_{i_{n-1}})$ for some $0\leq i_0<i_1<\ldots<i_{n-1}$. Then the ideals  of $n\times n$ minors of  $M$ and $N$, denoted by $I_{n}\left(M\right)$ and $I_{n}\left(N\right)$ respectively, and the defining ideal of the singular locus of $\A$ are related by 
\[
 I_{n}(M)\subseteq\sqrt{ I_{n}(N)}= J(\A).
 \]
 
 Further let $\theta_0,\ldots \theta_{n}$ be a basis for $\Der_R^G$ with coefficient matrix $Q=Q(\theta_0,\ldots,\theta_n)=\begin{bmatrix} \theta_j(x_i) \end{bmatrix}_{1\leq i,j \leq n}$ and let $C$ be the $(n+1) \times n$ submatrix $C=Q(\theta_{i_0},\ldots,\theta_{i_{n-1}})$ for indices $0\leq i_0<i_1<\ldots<i_{n-1}$. Then the ideals  of $n\times n$ minors of  $Q$ and $C$, denoted by $I_{n}\left(Q\right)$ and $I_{n}\left(C\right)$ respectively, and the defining ideal of the singular locus of $\A$ are related by 
\[
 I_{n}(C)\subseteq \sqrt{ I_{n}(Q)}= J(\A).
 \]

\end{cor}
\begin{proof}
For the claimed equality, it suffices to argue at the level of the respective varieties that $V(I_{n}(N))= V(J(\A))=V(I_n(Q))$. Using the relationship between the rank and nullity, Lemmas \ref{lem:Steinberg} translates as follows 
\begin{eqnarray*}
V(I_{n}(M)) &=& \{p\in \P^n \mid \text{rank of } N \text{ at } p \text{ is at most }n-1\}\\
                   &=& \{p\in \P^n \mid \text{nullity of } N \text{ at } p \text{ is at least }2\}\\
		&=& \{p\in \P^n \mid \text{at least } 2 \text{ hyperplanes of } \A \text{ pass through } p\}\\
		&=& V(J(\A)).
\end{eqnarray*}
The same proof applies to show $V(J(\A))=V(I_N(Q))$ using Lemma \ref{lem:Steinbergcoeff}.
Lastly, the containments $I_{n}(M)\subseteq I_{n}(N) \subseteq \sqrt{I_{n}(N)}$ and $I_{n}(C)\subseteq I_{n}(Q) \subseteq \sqrt{I_{n}(Q)}$ complete the proof  of the claim.
\end{proof}

The general strategy of showing that equality holds in the above containments is given by the following result.

\begin{lem}[{\cite[Lemma 8]{Eng}}]
\label{lem:eqmult}
Let $I\subseteq J$ be two unmixed ideals having the same height. Then $I=J$ if and only if $I$ and $J$ have the same multiplicity.
\end{lem}

We shall apply this for the ideals  satisfying the containments 
\[
I_{n}\left(Q(\theta_0,\ldots,\theta_{n-1})\right)\subseteq J(\A) \text{ and } I_{n}\left(\Jac(f_0,\ldots,f_{n-1})\right)\subseteq J(\A)
\]
 of Corollary \ref{cor:InMsubsetJ}. Since $J(\A)$ is a union of linear subspaces of $\P^n$, the multiplicity $e(R/J(\A))$ is simply the number of these linear spaces, i.e. the number of codimension two flats in the intersection lattice $\L(\A)$. The following lemma will provide to be the crucial ingredient in computing the multiplicities of  $I_{n}\left(Q(\theta_0,\ldots,\theta_{n-1})\right)$ and $ I_{n}\left(\Jac(f_0,\ldots,f_{n-1})\right)$, which only depend on the degrees of $\theta_0,\ldots,\theta_{n-1}$ and $f_0,\ldots,f_{n-1}$ respectively.

\begin{lem}
\label{lem:multiplicity}
Suppose $M$ is an $n\times (n+1)$ matrix with homogeneous entries of degree $e_i$ in row $i$ and set $s=\sum_{i=1}^n e_i$. If $\het(I_n(M))=2$, then the multiplicity of $R/\left(I_n(M)\right)$ is $e\left(R/I_n(M)\right)=\sum_{i=1}^{n}\binom{s+e_i}{2}-(n+1)\binom{s}{2}$.
\end{lem}
\begin{proof}
By the Hilbert-Burch theorem, the graded minimal free resolution of $R/I_n(M)$  is
$$
0 \longrightarrow  \bigoplus_{i=1}^{n}R(-s-e_i) \stackrel{M}{\longrightarrow} R^{n+1}(-s) \longrightarrow R \longrightarrow R/I_n(M) \longrightarrow 0.
$$
 It follows that there is an equality of Hilbert series 
 \[HS(R/I_n(M))=HS(R)-HS\left(R^{n+1}(-s)\right)+HS\left(\oplus_{i=1}^{n}R(-s-e_i)\right).\]  Thus we deduce that \[HS(R/I_n(M))= \frac{1}{(1-t)^{n+1}}-(n+1)\frac{t^s}{(1-t)^{n+1}}+\sum_{i=1}^n\frac{t^{s+e_i}}{(1-t)^{n+1}}.\]  Since $\dim\left(R/I_n(M)\right)=n-1$, it follows that $HS(R/I_n(M))=\frac{h(t)}{(1-t)^{n-1}}$ for some polynomial $h(t)$.  Thus it follows that
 \[\frac{h(t)}{(1-t)^{n+1}}= \frac{1}{(1-t)^{n+1}}-(n+1)\frac{t^s}{(1-t)^{n+1}}+\sum_{i=1}^n\frac{t^{s+e_i}}{(1-t)^{n+1}}\] whence $(1-t)^2h(t)=1-(n+1)t^s+\sum_{i=1}^nt^{s+e_i}$.  Differentiating twice with respect to $t$ and evaluating at $t=1$ yields $e\left(R/I_n(M)\right)=h(1)=\sum_{i=1}^{n}\binom{s+e_i}{2}-(n+1)\binom{s}{2}$, proving the lemma.
\end{proof}

\subsection{Infinite families}

Next we proceed to a case by case analysis of the groups in the Shephard-Todd classification, with the goal of proving Theorem \ref{thm:eqJ} in each case. To begin, we treat the infinite family in the Shephard-Todd classification, namely the groups  $G(m,p,n)$ parametrized by triples of positive integers $m,n,p\in\N$ with $p\mid m$. The group $G(m, p, n)$ is the semidirect product of the abelian group of order $mn/p$ whose elements are $(\xi^{a_1},\xi^{a_2}, \ldots,\xi^{a_n})$, with $\xi$ is a primitive $m$-th root of unity and $\sum a_i\equiv 0 \pmod{p}$, by the symmetric group acting by permutations of the coordinates. 

The reflection arrangement $\mathcal{A}(G(m,m,n))$ consists of the hyperplanes defined by polynomials of the form $x_i-\xi x_j$, where $\xi$ is a primitive $m$-th root of unity and $1\leq i,j\leq n$.  The reflection arrangement $\mathcal{A}(G(m,1,n))$ consists of the arrangement $\mathcal{A}(G(m,m,n))$ along with the coordinate hyperplanes defined by $x_i$ where $1\leq i\leq n$. 

If $m=1$ then the only irreducible groups in this family are the symmetric groups $A_n=G(1,1,n+1)$. 
 We treat the case of the symmetric group separately since, unlike the other irreducible complex reflection groups, the rank of these groups is smaller than the dimension of the space they naturally act on.

\begin{prop}[Symmetric groups]
\label{prop:symmetric}
Let $G=A_{n}$ and consider the following matrices 
 $$M=\left(\begin{matrix}1&x_0&x_0^2&\cdots&x_0^{n-1}\\
\vdots& & & & \vdots\\
1&x_n&x_n^2&\cdots&x_n^{n-1}
\end{matrix}\right), 
M'=\left(\begin{matrix}
x_1-x_0&x_1^2-x_0^2&\cdots&x_1^{n-1}-x_0^{n-1}\\
\vdots & & & \vdots\\
x_n-x_0&x_n^2-x_0^2&\cdots&x_n^{n-1}-x_0^{n-1}
\end{matrix}\right).
$$
Then the reduced singular locus of $\A(A_{n})$ is defined by 
\begin{equation*}
\label{eq:VdM}
J\left(\A(A_{n})\right)=I_{n}(M)=I_{n-1}(M')=\left(\prod_{0\leq i< j\leq n, i,j\neq s}(x_i-x_j) \mid 1\leq s\leq n\right).
\end{equation*}
\end{prop}

\begin{proof}
The group $A_{n}$ is the symmetric group on $n+1$ elements (which has rank $n$), whence 
$$\A(A_{n})=V\left(\prod_{0\leq i< j\leq n} (x_i-x_j)\right)$$ and the basic invariants for this group can be taken to be  $f_i=x_0^i+x_1^i+\cdots+x_n^i$ with $1\leq i\leq n+1$. Let $J=J(\A_n)$. The matrix $M$ considered in this proposition is the jacobian matrix of the lowest degree $n+1$ basic invariants, so the containment $I_{n}(M)\subseteq J$ follows from Corollary \ref{cor:InMsubsetJ}.
The $n\times n$ minors of $M$ obtained by removing one column at a time are Vandermonde matrices leading to the description
\[
I_{n}(M)=\left(\prod_{0\leq i< j\leq n, i,j\neq s}(x_i-x_j) \mid 1\leq s\leq n\right).
\]
Since $M'$ is obtained from $M$ by elementary column operations followed by removing a row and column which are unit vectors, we have the identity $I_{n}(M)=I_{n-1}(M')$, which yields the containment  $I_{n}(M)=I_{n-1}(M')\subseteq J$. To see that the containment is truly an equality, we note that Lemma \ref{lem:multiplicity} with $e_1=1,\ldots, e_{n-1}=n-1$ and $s=\binom{n}{2}$ yields

\begin{eqnarray*}
e\left(R/I_n(M')\right) &= &\sum_{i=1}^{n-1}\binom{s+i}{2}-(n+1)\binom{s}{2} = \sum_{i=1}^{n-1}\left(\binom{s+i}{2} -\binom{s}{2}\right)-\binom{s}{2}\\
&=&\sum_{i=1}^{n-1} \frac{i(2s+i-1)}{2}-\binom{s}{2} = s^2+ \frac{\sum_{i=1}^{n-1} i^2}{2}-\frac{s}{2}-\binom{s}{2}\\
&=&\frac{s^2}{2}+\frac{(n-1)n(2n-1)}{12}=\frac{(n-1)n(n+1)(3n-2)}{24}.
\end{eqnarray*}

The equality  
$$I_{n}(M)=J=\bigcap\limits_{i\neq j, k\neq l, |\{i,j,k,l\}|\geq 3}(x_i-x_j,x_k-x_l)$$
 follows from Lemma \ref{lem:eqmult} by observing that the number of linear associated primes of $J$, namely $\Ass(J)=\{(x_i-x_j,x_k-x_l)|\{i,j,k,l\}\geq 3, i\neq j,k\neq l\}$ can be counted as follows. For each (unordered) set of four distinct indices $\{i,j,k,l\}$ we can form 3 distinct ideals in $\Ass(J)$, namely  $(x_i-x_j, x_k-x_l), (x_i-x_k, x_j-x_l), (x_i-x_l, x_j-x_k)$.
For each set of three distinct indices $\{i,j,k\}$ we can form only one ideal in $\Ass(J)$ since by repeating any of the indices we get the same ideal $(x_i-x_j, x_i-x_k)=(x_i-x_j, x_k-x_j)=(x_i-x_k, x_j-x_k)$.
  Thus
 \begin{eqnarray*}
e(R/J) &=&3\binom{n+1}{4}+\binom{n+1}{3}\\
&=&\frac{(n+1)n(n-1)(n-2)}{8}+\frac{(n+1)n(n-1)}{6}\\
&=& \frac{(n+1)n(n-1)(3n-2)}{24}.
 \end{eqnarray*}
  \end{proof}

 We now turn our attention to the other irreducible groups in the infinite family of the Shephard-Todd clasification. Consider now $m\geq 2$ and focus on two subfamilies, namely the {\em monomial groups} $G(m,m,n+1)$ and the {\em full monomial groups} $G(m,1,n+1)$ with corresponding hyperplane arrangements
\begin{eqnarray*}
   \A(G(m,1,n+1))&=&V\left(x_0\cdots x_n\cdot \prod_{0\leq i<j\leq n} (x_i^m-x_j^m)\right)\\
    \A(G(m,m,n+1))&=&V\left( \prod_{0\leq i<j\leq n} (x_i^m-x_j^m)\right).
   \end{eqnarray*}
  If $p<m$ then $\A(G(m,p,n+1)=\A(G(m,1,n+1)$ by \cite[p.~247]{OT}, so in fact the two classes of hyperplane arrangements describes above exhaust all the reflection arrangements coming from this infinite family.
 We now describe the equations of the singular locus for each of them.
  
 \begin{prop}[Monomial groups]
 \label{prop:G(m,m,n)}
Let $G=G(m,m,n+1)$ with $m\geq 2$ and consider 
$$M=\left(\begin{matrix}x_1x_2\cdots x_n& x_0^{m-1}&x_0^{2m-1}&\dots&x_0^{(n-1)m-1}\\
\vdots&\vdots &\vdots & & \vdots\\
x_0x_1\cdots x_{n-1}&x_n^{m-1}&x_n^{2m-1}&\dots&x_n^{(n-1)m-1}
\end{matrix}\right).$$
Then the reduced singular locus of $\A(G(m,m,n+1))$ is defined by 
\begin{equation*}
\label{eq:G(m,m,n)}
J\left(\A(G(m,m,n+1))\right)=I_{n}(M)=\left(x_s\prod_{i,j\neq s, i\neq j}(x_i^m-x_j^m) \mid 0\leq s\leq n\right).
\end{equation*}
\end{prop}

\begin{proof}The basic invariants for the group $G(m,m,n)$ are the elementary symmetric polynomials in $x_i^m$,  $f_{d}=\sum_{i=0}^n x_i^{md}$ with $d=1,\ldots, n$, as well as $f_0=x_0\cdots x_n$. One sees at once that $M$ is the Jacobian matrix of the invariant polynomials $f_0,\ldots,f_{n-2},f_{n-1},f_n$.

Consider the submatrix of $M$ obtained by removing the $(s+1)$-st row corresponding to the variable $x_s$. Multiplying the $i$-th row of this matrix by $x_{i-1}$ if $i\leq s$ and by $x_i$ if $i>s$ followed by dividing the first column  by $x_0\cdots \widehat{x_s}\cdots x_n$  results in the following matrix having the same determinant
 \[M'=\left(\begin{matrix}x_s& x_0^{m}&x_0^{2m}&\dots&x_0^{(n-1)m}\\
\vdots& & & & \vdots\\
x_s&x_n^{m}&x_n^{2m}&\dots&x_n^{(n-1)m}
\end{matrix}\right).\]  
Therefore
\[
I_n(M)=\left(x_s\prod_{i,j\neq s, i\neq j}(x_i^m-x_j^m) \mid 0\leq s\leq n\right).
\]

Let $J=J(\A(G(m,m,n)))$, and let $P$ be an associated prime of $J$.  Then either $P=(x_a-\xi x_b,x_c-\sigma x_d)$ where $\{a,b\}\neq\{c,d\}$ and $\xi,\sigma$ are $m$-th roots of unity, or $P=(x_a,x_b)$ for some $a\neq b$.  Counting these primes it follows that 
\begin{eqnarray*}
e(R/J) &=& m^2\left(\binom{n+1}{3}+3\binom{n+1}{4}\right)+\binom{n+1}{2}\\
&=&\frac{m^2 n^4}{8} - \frac{m^2 n^3}{12} - \frac{m^2 n^2}{8} + \frac{m^2 n}{12} + \frac{n^2}{2} + \frac{n}{2}.
\end{eqnarray*}

On the other hand, by Lemma \ref{lem:multiplicity} with \[s=m(1+\dots+(n-1))-(n-1)+n=\frac{mn(n-1)}{2}+1,\] it follows that \[e(R/I_n(M))=\binom{n+\frac{mn(n-1)}{2}+1}{2}+ \sum_{i=1}^{n-1}\binom{im+\frac{mn(n-1)}{2}}{2}-(n+1)\binom{\frac{mn(n-1)}{2}+1}{2}, \text{ so}\] 
\[e(R/I_n(M))=\frac{m^2 n^4}{8} - \frac{m^2 n^3}{12} - \frac{m^2 n^2}{8} + \frac{m^2 n}{12} + \frac{n^2}{2} + \frac{n}{2}.\]
Thus $J(\A(G(m,m,n)))=I_n(M)$ by Lemma \ref{lem:eqmult} and this ideal is defined by the equations \eqref{eq:G(m,m,n)}.
 \end{proof}

  \begin{prop}[Full monomial groups]
  \label{prop:G(m,1,n)}
Let $G=G(m,1,n+1)$ with $m\geq 2$ and consider the matrix
$$M=\left(\begin{matrix}x_0&x_0^{m+1}&x_0^{2m+1}&\cdots&x_0^{(n-1)m+1}\\
\vdots& \vdots & \vdots & & \vdots\\
x_n&x_n^{m+1}&x_n^{2m+1}&\cdots&x_n^{(n-1)m+1}
\end{matrix}\right).$$
Then the reduced singular locus of $\A(G(m,1,n+1))$ is defined by 
\begin{equation*}
 \label{eq:G(m,1,n)}
J\left(\A(G(m,1,n+1))\right)=I_{n}(M)=\left(x_0x_1\cdots \widehat{x_s}\cdots x_n\prod_{i,j\neq s, i\neq j}(x_i^m-x_j^m) \mid 0\leq s\leq n\right).
\end{equation*}
\end{prop}

\begin{proof}
The basic invariants for $G(m,1,n+1)$ are the elementary symmetric polynomials in $x_i^m$,  $f_{d}=\sum_{i=0}^n x_i^{md}$ with $d=1,\ldots, n$, as well as $f_0=(x_0\cdots x_n)^m$. 
One sees that  the Jacobian matrix of the invariant polynomials $f_1,\ldots,f_{n}$ is

$$M'=\Jac(f_1,\ldots,f_{n})=\left(\begin{matrix}x_0^{m-1}& x_0^{2m-1}& \cdots & x_0^{nm-1}\\
x_1^{m-1} & x_1^{2m-1}& \cdots & x_1^{nm-1}\\
\vdots& \vdots & & \vdots\\
x_n^{m-1} & x_n^{2m-1} &\cdots&x_n^{nm-1}
\end{matrix}\right).$$

Factoring out one variable from each row of a maximal minor of $M$ yields a Vandermonde determinant, hence we deduce
\[
 I_n(M)=\left(x_0x_1\cdots \widehat{x_s}\cdots x_n\prod_{i,j\neq s, i\neq j}(x_i^m-x_j^m) \mid 0\leq s\leq n\right).
 \]
Similarly, factoring out the $(m-1)$-st power of a variable from each row of a maximal minor of $M'$ yields a Vandermonde determinant, hence we deduce
\[
 I_n(M')=\left(x_0^{m-1}x_1^{m-1}\cdots \widehat{x_s^{m-1}}\cdots x_n^{m-1}\prod_{i,j\neq s, i\neq j}(x_i^m-x_j^m) \mid 0\leq s\leq n\right).
\]
This yields the containment $I_n(M)\subseteq\sqrt{I_n(M')}=J(\A(G(m,1,n+1))$, where the last equality is given by Corollary \ref{cor:InMsubsetJ}.

 Let $J=J(G(m,1,n+1))$, and let $P$ be an associated prime of $J$.  Then either $P=(x_a,x_b)$, $P=(x_a,x_c-\sigma x_d)$, or $P=(x_a-\xi x_b,x_c-\sigma x_d)$ where $\{a,b\}\neq\{c,d\}$ and $\xi,\sigma$ are $m$-th roots of unity.  There are  $\binom{n+1}{2}$ primes of the form $P=(x_a,x_b)$, $m(n+1)\binom{n}{2}$ primes of the form $(x_a,x_c-\sigma x_d)$, and $3m^2\binom{n+1}{4}+m^2\binom{n+1}{3}$ primes of the form $P=(x_a-\xi x_b,x_c-\sigma x_d)$.  Thus
 \begin{eqnarray*}
 e(R/J) &=& \binom{n+1}{2}+m(n+1)\binom{n}{2}+3m^2\binom{n+1}{4}+m^2\binom{n+1}{3}\\
  &=& \frac{m^2 n^4}{8} - \frac{m^2 n^3}{12} - \frac{m^2 n^2}{8} + \frac{m^2 n}{12} + \frac{m n^3}{2} - \frac{m n}{2} + \frac{n^2}{2} + \frac{n}{2}.
  \end{eqnarray*}
 
 On the other hand, by Lemma \ref{lem:multiplicity} with $s=1+m(1+\dots+(n-1)+(n-1)=\frac{mn(n-1)}{2}+n$ it follows that 
 \[e(R/I_n(M))=\binom{1+n+\frac{nm(n-1)}{2}}{2}+\sum_{i=1}^{n-1}\binom{1+n+im+\frac{nm(n-1)}{2}}{2}+\binom{n+\frac{nm(n-1)}{2}}{2}\]
and a computation shows $e(R/I_n(M))=e(R/J)$. The equality $J=I_n(M)$ now follows from Lemma \ref{lem:eqmult}.
 \end{proof}

 \subsection{Sporadic groups}
 Finally we consider the sporadic irreducible complex pseudoreflection groups.

\begin{prop}
\label{prop:sporadictrue}
If $G$ is one of the pseudoreflection groups numbered $G_{23}$, $G_{24}$, $G_{27}$, $G_{28}$, $G_{29}$, $G_{30}$, $G_{33}$, $G_{34}$, $G_{35}$, $G_{36}$, $G_{37}$ in the  Shephard-Todd classification %
then $ J(\A(G))=I_{n}(M)$, where $M=\Jac(f_0,\ldots, f_{n-1})$ is the jacobian matrix of the $n=\rank(G)-1$ basic invariants of  lowest degree for $G$.

If $G$ is one of the pseudoreflection groups numbered $G_{23}$, $G_{25}$, $G_{26}$, $G_{28}$,  $G_{30}$, $G_{31}$, $G_{32}$, $G_{35}$, $G_{36}$, $G_{37}$ in the  Shephard-Todd classification then $ J(\A(G))=I_{n}(Q)$, where $Q=Q(\theta_0,\ldots, \theta_{n-1})$ is the coefficient matrix of the $n=\rank(G)-1$ basic derivations of  lowest degree for $G$.

Moreover, if $G$ is one of the pseudoreflection groups numbered $G_{23}$, $G_{28}$, $G_{30}$, $G_{35}$, $G_{36}$, $G_{37}$ then 
\[I_n(\Jac(f_0,\ldots,f_n))=I_n(M)=I_n(Q)=I_n(Q(\theta_0,\ldots,\theta_n)).\]
\end{prop}
\begin{proof}
Let $J=J(\A(G))$. By Corollary \ref{cor:InMsubsetJ} we have $I_n(M)\subseteq J$ and $I_n(Q)\subseteq J$. The Hilbert-Burch theorem and the definition respectively yield that $I_n(M), I_n(Q)$ and $J$ are unmixed ideals of the same height. Thus by Lemma \ref{lem:eqmult} it suffices to show that the multiplicities agree, i.e. either $e(R/J)=e\left(R/I_n(M)\right)$ or $e(R/J)=e\left(R/I_n(Q)\right)$ depending on the group.
The multiplicity $e(R/J)$ is the number of flats of codimension two in the intersection lattice of $\A(G)$, which can be deduced from \cite[Tables C.5--C.23]{OT}. Moreover Table B.1 in \cite{OT} contains information regarding the exponents and coexponents of each irreducible complex reflection group as rendered below. Lemma \ref{lem:multiplicity} allows to compute the multiplicities $e\left(R/I_n(M)\right)$ and $e\left(R/I_n(Q)\right)$ in terms of the exponents $\deg(f_i)-1$ and coexponents $\deg(\theta_i)-1$ for $G$, which are the degrees of the polynomials in each column of $M$ and $Q$ respectively. We use the symbol ---''--- to indicate that the exponents and coexponents coincide for a specific group. These considerations yield the following data, where the columns labeled $e_M, e_Q$ record  $e\left(R/I_n(M)\right)$ and $e\left(R/I_n(Q)\right)$ respectively.

\begin{equation}
\label{table}
{\small
\begin{tabular}{c|c|c|c|c|c}
Group  &Exponents & Coexponents & $e_M$& $e_Q$& $e(R/J)$\\
\hline
$G_{23}$ & 1, 5, 9 & ---''--- & 31 &31 & 31 \\
$G_{24}$ &3, 5, 13  & 1, 9, 11 &49 & 91&49 \\
$G_{25}$ &5, 8, 11  & 1, 4, 7 & 129 & 21&21\\
$G_{26}$ & 5, 11, 17& 1, 7, 13  &201  & 57 &57 \\
 $G_{27}$ & 5, 11, 29 & 1, 19, 25 &201& 381&201\\ 
 $G_{28}$ & 1, 5, 7, 11 & ---''--- & 122 & 122&122\\ 
 $G_{29}$ & 3, 7, 11,19 & 1, 9, 13, 17 & 310 &390 &310\\
 $G_{30}$ & 1, 11, 19, 29 & ---''--- &722 & 722& 722 \\
 $G_{31}$ & 7, 11, 19, 23 &1, 13, 17, 29 &950 & 710& 710 \\
 $G_{32}$ & 11, 17, 23, 29 & 1, 7, 13, 19 &1770 & 330 & 330 \\
 $G_{33}$ & 3, 5, 9, 11, 17 & 1, 7, 9, 13, 15&510 & 600& 510 \\
 $G_{34}$ & 5, 11, 17, 23, 29, 41 & 1, 13, 19, 25, 31, 37&4515 & 5019&4145 \\
 $G_{35}$ & 1, 4, 5, 7, 8, 11 & ---''---&390 & 390& 390\\
 $G_{36}$ & 1, 5, 7, 9, 11, 13, 17 & ---''--- &1281 & 1281&1281 \\
  $G_{37}$ & 1, 7, 11, 13, 17, 19, 23, 29 & ---''---&4900 & 4900&4900\\
\end{tabular}
}
\end{equation}
One can now check the ideal equalities in the first two claims follow from the equality of the respective multiplicities. The last claim follows from the first two claims and the containments in Corollary \ref{cor:InMsubsetJ}.
\end{proof}

\begin{rem}
Two particular cases of the previous proposition have already appeared in the literature, namely the equations of the singular points for the arrangements corresponding to $G_{24}$ and $G_{27}$ are determined in \cite{BDH}.
\end{rem}

From the previous results we assemble the proof of Theorem \ref{thm:eqJ}:

\begin{proof}[Proof of Theorem \ref{thm:eqJ}]
Follows from Propositions \ref{prop:symmetric}, \ref{prop:G(m,m,n)}, \ref{prop:G(m,1,n)}, \ref{prop:sporadictrue} and the Hilbert-Burch theorem \cite[Theorem 20.15]{Ei}.
\end{proof}

\section{Associated primes and localization}
\label{s:localization}

Our goal in the next section will be to consider the containment $J^{(3)}\subseteq J^2$ for ideals $J=J(\A)$ defining the singular locus of a reflection arrangement $\A$. This task is facilitated by the main results of this section: the determination of the associated primes of $J^2$ and a description of a notion of localization for hyperplane arrangements.

\subsection{Localization of hyperplane arrangements}

In proving containments and noncontainments of powers and symbolic powers of ideals, it can be helpful to consider localizations of those powers.  This section describes ways in which information about the structure of an arrangement transfers to information about the various localizations of ideals arising from it.

\begin{defn}
Given a flat $X$ in an arrangement $\mathcal{A}$, the {\em localization} of $\mathcal{A}$ at $X$ is the hyperplane arrangement $\mathcal{A}_X=\{H\in\mathcal{A} \mid H\supseteq X\}$.  
\end{defn}

A hyperplane arrangement $\A$ is termed {\em central} if $\bigcap_{H\in\A} (\ell_H)\neq (0)$. Notice that the localization $A_X$ is a central arrangement because $\bigcap_{H\in\A} H=X$. The rank of a hyperplane arrangement $\A$, $\rank(\A)$, is the dimension of the space spanned by the normals to the hyperplanes in $\A$. We say that $\A$ is \textit{essential} if the rank of $\A$ is equal to the vector space dimension of the ambient space. However, if $\A$ is central with $\bigcap_{H\in\A} H=X$, then $\rank(\A) = \codim(X)$, so $\A$ is not essential. Take $Y$ to be a complementary space in $\P^n$ to $X$, for example, $Y=\{v\in \P^n \mid \langle v,x\rangle =0, \forall x\in X\}$.
Since we have $\codim_Y (H \cap Y) = 1$ for all $H\in \A_X$, the set $\A'_X = \{H \cap Y \mid Y \in \A_X\}$ is an essential arrangement in $\P(Y)$. Moreover, the arrangements $\A_X$ and $\A'_X$ have isomorphic intersection posets. Let us call $\A'_X$ the \textit{essentialization} of $\A_X$, denoted $\ess(\A_X)$. 

This notion of localization for hyperplane arrangements relates to the algebraic notion of localization as follows. 
Let $P$ be the defining ideal of $X$ and choose a vector space $Q_1\subseteq R_1$ such that $P_1\oplus Q_1=R_1$. If one defines $Q$ to be the ideal generated by $Q_1$, then $Y=V(Q)$ is complementary to $X$ in $\P^n$. Consider the projection map away from $X$, $\pi_X:\P^n\to \P(Y)$ represented algebraically by the inclusion $\iota_X:k[\P(Y)]\cong \Sym(P_1)\hookrightarrow R$. Then the description above yields $\ess(\A_X)=\pi_X(\A_X)$ and thus $\iota_X(F_{\ess(\A_X)})=F_{\A_X}$. What is more, the localized arrangement can be obtained from the original arrangement by localization at $P$ as 
\begin{equation}
\label{eq:localF}
F_{\ess(\A_X)}R_P=F_\A R_P.
\end{equation}

We now explain how localization can be related to the group governing a reflection a\-rrange\-ment.
\begin{defn}
Given a flat $X$ in a reflection arrangement $\mathcal{A}(G)$, the {\em fixer} of $X$ is the subgroup 
$G_X=\{g\in G \mid g(x)=x\mbox{ for all }x\in X\}$.  
\end{defn}
While not all subgroups of a reflection group are reflection groups themselves, Steinberg has shown \cite[Theorem 1.5]{Steinberg2} that fixers of flats in $\L(\A)$ are reflection groups. Therefore it makes sense to consider the arrangement $\A(G_X)$. By definition this arrangement has as ambient space a vector space of dimension $\rank(G_X)$, hence the arrangement satisfies $\rank(\A(G_X))=\rank(G_X)$ and is essential. The relationship between $\A(G_X)$ and $\A(G)_X$ is illustrated in the figure below and made precise in Lemma \ref{lem:arrangementlocalization}.

\begin{figure}[h!]
\begin{center}
\begin{tikzpicture}
\draw [dashed, ultra thick, blue] (0,0) -- (1.5,2.6);
\draw [ultra thick, blue] (0,0) -- (1.24,2.15);
\node [right, blue] at (.8,1.3) {$X$};
\draw [ultra thick, red] (2.5,2.6) -- (.5,2.6);
\draw [thick] (.5,2.6) -- (-1,0);
\draw [thick] (-1,0) -- (0,0);
\draw [thick] (0,0) -- (.866,-.5);
\node [right] at (-.7,.75) {$\mathcal{A}(G)$};
\draw [thick] (.866,-.5) -- (2.366,2.1);
\draw [ultra thick, red] (2.333,2.1) -- (.9,2.946);
\node [right, red] at (2.1,2.9) {$\mathcal{A}(G_X)$};
\draw [thick] (.9,2.946) -- (.7,2.6) ;
\draw [thick] (2.5,2.6) -- (2.2424,2.153);
\draw (-.1,2.946) -- (1.9,2.946);
\draw (1.333,2.1) -- (-.1,2.946);
\draw (3.333,2.1) -- (1.9,2.946);
\draw (3.33,2.1) -- (1.33,2.1);
\draw [fill] (1.5,2.6) circle [radius=.06];
\draw [fill] (0,0) circle [radius=.06];
\end{tikzpicture}
\caption{Localization of a hyperplane arrangement}
\end{center}
\end{figure}
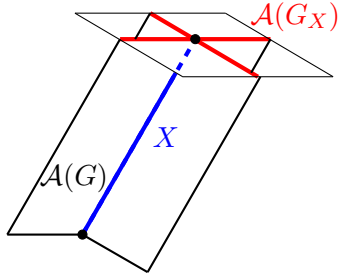

\begin{lem}
\label{lem:arrangementlocalization}
The following arrangements coincide $\A(G_X)=\ess(\A(G)_X)$.
\end{lem}

\begin{proof}
Since the action of a reflection on a flat fixes the flat point-wise if and only if the flat is contained in the hyperplane fixed by the reflection, it follows that 
\begin{eqnarray*}
\mathcal{A}(G)_X&=&\{H\in\A \mid H \text{ is fixed by } g, \forall g\in G_X\}, \text{ thus}\\
\ess(\mathcal{A}(G)_X)&=&\{H\cap Y \subseteq \P^{\rank(\A(G_X))-1} \mid H\cap Y \text{ is fixed by } g, \forall g\in G_X\}=\A(G_X).
\end{eqnarray*}
\end{proof}

The singular locus of a reflection arrangement localizes in a similar way to the localization of the defining equation of the arrangement presented in \eqref{eq:localF}.

\begin{lem} 
\label{lem:locJ}
Let $\mathcal{A}$ be a reflection arrangement, and let $X$ be any flat of $\A$ with $X=V(P)$. Then there is an identity $J(\mathcal{A})_P=J(\mathcal{A}_X)_P$.
\end{lem}
\begin{proof}
Recall that $J(\mathcal{A})=\bigcap_{H\neq H'\in \A}(\ell_H,\ell_{H'})$, thus we have
\begin{eqnarray*}
J(\mathcal{A})_P &=& \bigcap_{H\neq H'\in \A, \ell_H,\ell_{H'}\in P}(\ell_H,\ell_{H'})_P=\bigcap_{X\subseteq H\neq H'\in \A}(\ell_H,\ell_{H'})_P\\
&=& \bigcap_{H\neq H'\in \A_X}(\ell_H,\ell_{H'})_P=J(\mathcal{A}_X)_P. 
\end{eqnarray*}
\end{proof}

From this lemma, it follows that one can gain information on the singular loci of reflection arrangements by looking at the reflection arrangements of fixers of flats in $\L(\A)$. This makes up a complete picture of all the relevant localization because the associated primes of $J(\A)$ correspond to the codimension two flats in $\L(\A)$. 
The next lemma presents a similar picture for the associated primes for $J(\A)^2$, namely that they correspond to codimension three flats in $\L(\A)$. Consequently all the relevant localizations for $J(\A)^2$ are still given by reflection arrangements of subgroups of $G$ of rank 3.

\begin{lem}
\label{lem:locJ^2}
Let $J(\A)$ be the radical  ideal defining the singular locus of a reflection arrangement $\A=\A(G)$, with $\rank(G)\geq 4$. Then any associated prime $P\in \Ass(J(\A)^2)$ defines an essential flat $X$ of codimension 2 or 3 in the intersection lattice of $\A$ such that $J(\A)_P^2=J(\A_X)^2_P$, for $X=V(P)$.
\end{lem}
\begin{proof}
Let $J=J(\A)$ and let $R$ be the the ambient polynomial ring for $J$ of dimension $\dim(R)=\rank(G)$. That the associated primes of the powers $J^r$ are ideals defining flats follows for any $r\geq 1$ by \cite[Proposition 4.5]{DS}. It remains to see that the associated primes have the claimed codimension.

By Theorem \ref{thm:eqJ} $J$ is a perfect ideal of height two 
therefore it is linked to a complete intersection (licci). By \cite[Theorem 1.14]{H} licci ideals have Cohen-Macaulay Koszul homology for their generating sequences. Then \cite[Remark 2.10]{U} together with the fact that $J$ is a complete intersection when localized at any of its associated primes (all of which have height two) yields
\[
\depth(R/J^2)\geq \rank(G)-2-2+1=\rank(G)-3\geq 1.
\]
This gives that the homogeneous maximal ideal is not in  $\Ass(J^2)$ and furthermore, if $P\in \Ass(J^2)$ then $2=\height(J) \leq \height(P)\leq \rank(G)-\depth(R/J^2)\leq 3$.

\end{proof}

We note that the assertion $X\neq \emptyset$ is the conclusion of the previous Lemma is a crucial feature in our inductive approach because it implies that $\rank(G_X)<\rank(G)$. For arbitrary $r\geq 1$ the associated primes of $J(A)^r$ continue to define flats in the intersection lattice of $\A$ by \cite[Proposition 4.5]{DS}, but for $\rank(G)=3$ and $r\geq 2$ or more generally for $r\geq \rank(G)-1$ the homogeneous maximal ideal becomes an associated prime of $J(A)^r$. 

\subsection{Localization criteria for containments between powers}

We recall some well known properties of ordinary and symbolic powers with regard to localization, specifically  that localization commutes with taking powers and symbolic powers and its interplay with containments between ordinary and symbolic powers.

\begin{lem}
\label{localizeSymb}
Let $R$ be a Noetherian ring, let $I\subseteq R$ be an ideal of $R$, and let $m,r\in\mathbb{N}$.  Then
\begin{enumerate}
    \item$(I_P)^r=(I^r)_P$ for all $r\geq 0$ for all  $P\in\Spec R$.
    \item $(I_P)^{(m)}=(I^{(m)})_P$ for all $P\in\Spec R$.
    \item If $R$ is local or graded  with maximal ideal $\mathfrak{m}$ and  $I^{(m)}\not\subseteq I^r$, then $I_\mathfrak{m}^{(m)}\not\subseteq I_\mathfrak{m}^r$.
    \item If for some $P\in\Spec(R)$  there is a noncontainment $I_P^{(m)}\not\subseteq I_P^r$, then $I^{(m)}\not\subseteq I^r$.
    \item If the containment $I_P^{(m)}\subseteq I_P^r$ holds for all $P\in\Ass(I^r)$, then $I^{(m)}\subseteq I^r$.
\end{enumerate}
\end{lem}

The following lemma will also be very useful to us.
\begin{lem}[{\cite[Proposition 2.1]{Walk}}]
\label{walker}
Let $S$ and $R$ be commutative rings and $\phi:S\rightarrow R$ be a faithfully flat map.  Let $I\subseteq S$ be an ideal satisfying the property that $\phi(P)$ is prime for all $P\in\Ass(I)$.  Then for each $m,r\in\mathbb{N}$ $I^{(m)}\subseteq I^r$ if and only if $\phi(I)^{(m)}\subseteq\phi(I)^r$. 
\end{lem}
Using this together with the preceding lemmas, we obtain the following technical statement which shall be useful for our purposes.

\begin{prop}
\label{diagram}
Let $R$ and $S$ be commutative Noetherian local or graded rings and let  $\phi:S\rightarrow R$ be a flat map.  Denote by  $\mathfrak{m}$ the (homogeneous) maximal ideal of $S$ and assume that $\phi(\fm)=P$ is a prime ideal of $R$.  Let $I\subseteq S$ and $J\subseteq R$ be (homogeneous) ideals such that $\phi(I)_P=J_P$ and $\phi(Q)$ is prime for each $Q\in\Ass(I)$.  If for some $m,r\in\N$ there is a non-containment $I^{(m)}\not\subseteq I^r$, then $J^{(m)}\not\subseteq J^r$.
\end{prop}
\begin{proof}
The universal property of the localization $S_\fm$ yields a homomorphism $\phi_\mathfrak{m}:S_\mathfrak{m}\rightarrow R_P$ given by $\frac{a}{b}\mapsto \frac{\phi(a)}{\phi(b)}$, which fits into the following commutative diagram:
$$
\begin{CD}
S     @>\phi>>  R\\
@VVV        @VVV\\
S_\fm     @>\phi_\fm>>  R_P
\end{CD}
$$
What is more, as  flat map between local rings, $\phi_\mathfrak{m}$ is faithfully flat.  Since $\Ass(I_\fm)=\{qS_\fm \mid q\in \Ass(I)\}$ and $\phi(Q)\in \Spec(R)$ for each $Q\in \Ass(I)$ one concludes that $\phi_\fm(\Ass(I_m))\subseteq \Spec(R_P)$. Since $I^{(m)}\not\subseteq I^r$, Lemma \ref{localizeSymb} (3) yields $I_\mathfrak{m}^{(m)}\not\subseteq I_\mathfrak{m}^r$ and since $\phi_\mathfrak{m}$ is faithfully flat, Lemma \ref{walker} further gives $\phi_\fm(I)^{(m)}\not \subseteq \phi_\fm(I)^{r}$. In view of the commutativity of the above diagram we deduce that $\phi_\fm(I_\fm)=\phi(I)_P=J_P$ and thus the non-containment above can be rewritten as $J_P^{(m)}\not\subseteq J_P^r$. Therefore, by Lemma \ref{localizeSymb} (4), we conclude that $J^{(m)}\not\subseteq J^r$.
\end{proof}

Finally we are able to assemble our localization techniques into a criterion for (non)containment between ordinary and symbolic powers.

\begin{thm}
\label{thm:localizationcriterion}
Let $\A(G)$ be a reflection arrangement and let $m,r$ be positive integers.
\begin{enumerate}
\item If for all $P\in \Ass(R/J(\A(G))^r)$ there is a containment $J(\A(G))_P^{(m)}\subseteq J(\A(G))_P^{r}$, then $J(\A(G))^{(m)}\subseteq J(\A(G))^{r}$. In particular, if for every codimension 2 or 3 flat $X\in\L(\A)$ one has $J(\A(G_X))_P^{(3)}\subseteq J(\A(G_X))_P^{2}$, then $J(\A(G))^{(3)}\subseteq J(\A(G))^{2}$.
\item If there is $P\in \Ass(R/J(\A(G))^r)$ such that $J(\A(G))_P^{(m)}\not \subseteq J(\A(G))_P^{r}$, then $J(\A(G))^{(m)}\not \subseteq J(\A(G))^{r}$. In particular, if there is a codimension three  flat $X\in\L(\A)$ such that $J(\A(G_X))_P^{(3)}\not\subseteq J(\A(G_X))_P^{2}$, then $J(\A(G))^{(3)}\not\subseteq J(\A(G))^{2}$.
\end{enumerate}
\end{thm}
\begin{proof}
(1) The first statement follows from part (5) of Lemma \ref{localizeSymb}. 
To apply this to the specific case of reflection arrangements, let $P$ be  an associated prime for $J(\A(G))^{2}$ and let $X$ be the flat of $\mathcal{A}(G)$ defined by $P$ according to Lemma \ref{lem:locJ^2}. By Lemma \ref{lem:arrangementlocalization}  $\A(G_X)$ is the essentialization of the arrangement $\A(G)_X$. Furthermore by the remarks preceding that lemma, there is an inclusion $\iota_X: S\hookrightarrow R$, where $S=\Sym(P_1)$ is the coordinate ring of the ambient space of  $\A(G_X)$. Note that $P$ is the maximal ideal of $S$ and the inclusion map $\iota_X$ is flat and maps prime ideals to prime ideals.  Thus $\iota_X$ induces a faithfully flat map $(\iota_X)_P:S_P\rightarrow R_{P}$ which also maps prime ideals to prime ideals and satisfies by Lemma \ref{lem:locJ}
\[
(\iota_X)_P(J(\mathcal{A}(G_X))_P)=J(\mathcal{A}(G_X))_P=J(\mathcal{A}(G))_{P}.
\]
The hypothesis grants $J(\mathcal{A}(G_X))^{(3)}\subseteq J(\mathcal{A}(G_X))^2$, whence we deduce that $J(\mathcal{A}(G_X))_P^{(3)}\subseteq J(\mathcal{A}(G_X))_P^2$. By Lemma \ref{lem:locJ} this can be transcribed as $J(\A(G))_{P}^{(3)}\subseteq J(\A(G))_{P}^2$.  Thus, by Lemma \ref{localizeSymb} (5), since this containment holds when localized all every associated prime of $J(\A(G))^2$, one can deduce that $J(\A(G))^{(3)}\subseteq J(\A(G))^2$.

(2) The first statement follows from part (4) of Lemma \ref{localizeSymb}.  In view of Lemma \ref{lem:locJ^2}, any associated prime $P$ for $R/J(\A(G))^2$ is the defining ideal of a  flat $X\in \L(\A)$ such that $\codim(X)\in\{2,3\}$. By Lemma \ref{lem:locJ}, under the hypothesis of the latter statement  we have the noncontainment 
\[
J(\A(G))_P^{(3)}=J(\A(G_X))_P^{(3)}\not \subseteq J(\A(G_X))_P^{2}=J(\A(G))_P^{2},
\]
and thus the second statement follows from the former.
\end{proof}

\section{Symbolic power containment in reflection arrangements}
\label{s:containment}

In this section we consider the question: for which reflection arrangements $\A(G)$ is the containment $J(\A(G))^{(m)}\subseteq J(\A(G))^r$ satisfied for a given pair of positive integers $m,r$? We give the most comprehensive answers in the case $m=3, r=2$.

The general strategy we follow goes along these lines: first, consider the decomposition of an arbitrary pseudoreflection group as the direct product of irreducible pseudoreflection groups and reduce the problem to checking the respective containments for each of the irreducible factors. Second, using the ideas of section \ref{s:localization}, the problem is reduced further to arrangements determined by fixers of flats, which settles the argument by an induction on the rank of the groups involved .

\subsection{Reduction to the irreducible case}

Let $G=G_1\times \cdots\times G_s$ be the product of reflection groups acting on $\P^{n_1},\ldots,\P^{n_s}$ respectively. Then $G$ acts on $\P^{n_1}\times \cdots \times \P^{n_s}$ in the obvious manner determining a reflection group denoted $\A(G)=\A(G_1)\times \cdots\times \A(G_s)$, whose defining polynomial is $F_{\A(G)}=\prod_{i=1}^sF_{\A(G_i)}$.
We start by establishing a formula for the singular locus of a product of reflection groups.

\begin{lem}
\label{lem:products2}
Let $G_1$ and $G_2$ be reflection groups with $\mathcal{A}_1=\mathcal{A}(G_1)=V(F_1)$ and $\mathcal{A}_2=\mathcal{A}(G_2)=V(F_2)$.  Then $J(\mathcal{A}(G_1\times G_2))=J(\mathcal{A}_1\times\mathcal{A}_2)=F_2J(\mathcal{A}_1)+F_1J(\mathcal{A}_2)$.
\end{lem}

\begin{proof}
 Let $I_1=J(\mathcal{A}_1)$, $I_2=J(\mathcal{A}_2)$, and $I=F_2I_1+F_1I_2$. 
Since $F_1\in I_1$, it follows by the modular law 
that 
$$
I_1\cap(F_1,F_2)=I_1\cap (F_1)+I_1\cap (F_2)=(F_1)+I_1\cap (F_2).
$$
Since $F_2$ is expressed in a different set of variables than the generators of $I_1$, it follows that $I_1\cap (F_2)=F_2I_1$.  Similarly, since $F_2I_1\subseteq I_2$,
 $
 I_2\cap((F_1)+F_2I_1)=F_1I_2+F_2I_1.
 $
 Thus $I=I_1\cap I_2\cap (F_1,F_2)$.  By repeated application of the modular law, it follows that 

\[(F_1,F_2)=\bigcap_{H_1\in \A_1, H_2\in \A_2}(\ell_{H_1},\ell_{H_2}).\]
Overall this yields
\begin{eqnarray*}
I &=& \bigcap_{H_1,H_1'\in \A_1}(\ell_{H_1},\ell_{H'_1})\cap \bigcap_{H_2, H'_2\in \A_2}(\ell_{H_2},\ell_{H'_2})\cap \bigcap_{H_1\in \A_1, H_2\in \A_2}(\ell_{H_1},\ell_{H_2})\\
 &=& \bigcap_{H,H'\in \A_1\times \A_2}(\ell_{H},\ell_{H'})=J(\A_1\times\A_2).
\end{eqnarray*}

\end{proof}

The previous lemma generalizes to provide a closed formula for the singular locus of a product of multiple arrangements.

\begin{lem}
\label{lem:products}
Let $G_1,\ldots, G_s$ be reflection groups with $\mathcal{A}_i=\mathcal{A}(G_i)=V(F_i)$  Then the singular locus of the arrangement $\A(G_1\times \cdots \times G_s)$ is defined by the ideal
\[
J(\A(G_1\times \cdots \times G_s))=\sum_{i=1}^s\left(\prod_{1\leq j\neq i\leq s}F_j\right)J(\mathcal{A}(G_i)).
\] 
\end{lem}
\begin{proof}
The claim follows by induction on the number of factors, as
\[
J(\A(G_1\times \cdots \times G_s))=\sum_{i=1}^s\left(\prod_{1\leq j\neq i\leq s}F_j\right)J(\mathcal{A}(G_i)).
\] 
and the identity
\[
J(\A(G_1\times \cdots \times G_s\times G_{s+1}))= F_{s+1}J(\A(G_1\times \cdots \times G_s))+\left(\prod_{i=1}^sF_i \right)J(\A(G_{s+1})
\]
which follows from the previous Lemma, combine to give the claim for $s+1$ factors.
\end{proof}
We now present two lemmas which describe criteria for containments and non-containments between symbolic and ordinary powers for ideals of the type described in Lemma \ref{lem:products2}, that is, ideals $GI+FJ$ where $I$ and $J$ are extended from distinct rings, $F\in I$ and $G\in J$.

\begin{lem}
\label{claim}
Let $R,S$ be graded $k$-algebras with $I\subseteq R$ and $J\subseteq S$ homogeneous ideals.  Let $F\in I$ and $G\in J$, and suppose that for some $m$ and $r$, $I^{(m)}\not\subseteq I^r$.  Then the ideal $L=FJ+GI\subseteq T=R\otimes_k S$ satisfies $L^{(m)}\not\subseteq L^r$.
\end{lem}
\begin{proof}
Let $\mathfrak{m}$ be the homogeneous maximal ideal of $R$.  Since $I^{(m)}\not\subseteq I^r$, it follows that $I_\fm^{(m)}\not\subseteq I_\fm^r$ by Lemma \ref{localizeSymb}.  Since the natural inclusion $R\hookrightarrow T$ is flat, so is its localization $R_\fm\hookrightarrow T_\fm$.  Thus by Lemma \ref{walker} we have $I_\fm^{(m)}T_\fm\not\subseteq I_\fm^rT_\fm$ which implies that $(I^{(m)}T)_\fm\not\subseteq (I^rT)_\fm$.

By definition $L=IT\cap JT\cap (F,G)$ and $(F,G)$ is a complete intersection.  Since $\Ass(I)$, $\Ass(J)$, and $\Ass\left((F,G)\right)$ are pairwise disjoint, it follows that \[L^{(m)}=I^{(m)}T\cap J^{(m)}T\cap (F,G)^m,\] and thus 
$L^{(m)}_\fm = I^{(m)}T_\fm$.
Furthermore, since \[L^r=\sum_{i=0}^rF^iG^{r-i}J^iI^{r-i}\] it follows that \begin{equation}
\label{eq3}
L^r_\fm=\left(\sum_{i=0}^rF^iI^{r-i}\right)T_\fm=I^rT_\fm.
\end{equation}
Therefore, we see that $L_\fm^{(m)}\not\subseteq L_\fm^r$.  Thus by Lemma \ref{localizeSymb}, $L^{(m)}\not\subseteq L^r$.
\end{proof}

\begin{lem}
\label{claim2}
Let $R,S$ be graded $k$-algebras with $I\subseteq R$ and $J\subseteq S$ homogeneous ideals.  Let $F\in I^{(2)}$ and $G\in J^{(2)}$, and suppose that for some $r$, $I^{(2r-1)}\subseteq I^r$ and $J^{(2r-1)}\subseteq J^3$.  Then the ideal $L=FJ+GI\subseteq T=R\otimes_k S$ satisfies $L^{(2r-1)}\subseteq L^r$.
\end{lem}
\begin{proof}
As in the previous lemma, we have \[L^{(2r-1)}=I^{(2r-1)}\cap J^{(2r-1)}\cap(F,G)^{2r-1}.\]  

Let $f\in L^{(2r-1)}$.  Since $f\in (F,G)^{2r-1}$ it follows that \[f=u_0F^{2r-1}+u_1F^{2r-2}G^{1}+\dots+u_{2r-1}G^{2r-1}\] for some $u_0,\dots, u_{2r-1}$.  Since $G\in J^{(2)}$, one has that $F^{2r-i-1}G^i\in J^{(2i)}$ for each $0\leq i\leq 2r-1$, so in particular $u_iF^{2r-i-1}G^i\in J^{(2r-1)}$ for all $i\geq r$.  Thus we have \[\sum_{i=0}^{r-1} u_iF^{2r-i-1}G^i\in J^{(2r-1)}.\]  Factoring out $F^r$ from each term, we obtain \[\sum_{i=0}^{r-1} u_iF^{2r-i-1}G^i\in F^rJ^{(2r-1)}\subseteq F^rJ^{r}.\]  By a similar argument, it follows that \[\sum_{i=r}^{2r-1} u_iF^{2r-i-1}G^i\in G^rI^{(2r-1)}\subseteq G^rI^{r},\] whereby we may conclude that $f\in F^rJ^r+G^rI^r\subseteq (FJ+GI)^r$.  Therefore there is a containment \[(FJ+GI)^{(2r-1)}\subseteq (FJ+GI)^r.\] 
\end{proof}

The lemmas above give a criterion for containments between ordinary and symbolic powers of singular loci for products of hyperplane arrangements in terms of their factors.

\begin{lem}
\label{lem:containmentproducts}
Let $\A_1,\ldots, \A_s$ be hyperplane arrangements in distinct projective spaces with singular loci defined by ideals  $I_i=J(\mathcal{A}_i)$ and let $J_s=J(\A_1\times \cdots\times \A_s)$. For positive integers $m\geq r$
\begin{enumerate}
\item  if $I_i^{(m)}\not \subseteq I_i^r$ for some $0\leq i\leq s$ then $J_s^{(m)}\not\subseteq J_s^r$.
\item if $I_i^{(2r-1)} \subseteq I_i^r$ for all $0\leq i\leq s$ then $J_s^{(2r-1)}\subseteq J_s^r$.
\end{enumerate}
\end{lem}
\begin{proof}

We argue both statements simultaneously  by induction on $s$, with base case $s=1$ tautologically true.
Suppose the claims holds for $J_s$.  By Lemma \ref{lem:products} we know that \[J_{s+1}= F_{s+1}J_s+\left(F_1\dots F_s\right)I_{s+1}.\]  Then, the first claim holds by Lemma \ref{claim} and the second by Lemma \ref{claim2}.
\end{proof}

\subsection{Containment} 
Part (1) of Theorem \ref{thm:localizationcriterion} gives a criterion for showing that $J(\mathcal{A})^{(3)}\subseteq J(\mathcal{A})^2$ where $\mathcal{A}$ is a reflection arrangement.  To apply this, however, we must understand the structure of the associated primes for the square of the ideal defining the singular locus of $\A$,  $J(\mathcal{A})^2$, as well as the possible localizations of the reflection arrangement.  An important step in this direction is given by Lemma \ref{lem:locJ^2}, which shows that such associated primes correspond to flats $X$ of codimension 2 or 3 in $\L(\A)$. We take one step further and analyze the localization $J(\mathcal{A})_P$ for any prime $P$ such that $X=V(P)\in \L(\A)$ is nonempty. We call such a flat $X$, whose defining ideal is not the homogeneous maximal ideal $\fm$ of $R$ an essential flat.

Proposition \ref{lem:products} is the main tool used to describe the behavior of singular loci under localization.  To begin with, we consider the symmetric groups, monomial groups, and full monomial groups.

\subsubsection{Symmetric groups}
The arrangement $\mathcal{A}(A_n)$ consists of the hyperplanes $H_{i,j}$ with defining equations $x_i-x_j$ for all $i\neq j\in\{0,\dots,n\}$.  In the following we abuse notation by writing $H_{i,j}=x_i-x_j$.  We note that reflection across $H_{i,j}$ corresponds to the action of the transposition $(i,j)\in\mbox{Sym}(n+1)$ on $R$ given by $x_i\mapsto x_j$ and $x_j\mapsto x_i$.
\begin{lem}
\label{AnLoc}
Let $\mathcal{A}=\mathcal{A}(A_n)$, and let $X$ be an essential flat of $\mathcal{A}$.  Then \[\mathcal{A}_X=\mathcal{A}(A_{n_1}\times\dots\times A_{n_s})\] where for each $i$ there is an inequality $n_i<n$.
\end{lem}
\begin{proof}
We know that $\mathcal{A}_X$ consists of the hyperplanes in $\mathcal{A}$ containing $X$.  Define the reflexive relation $\sim$ on $\mathcal{A}_X$ where $H_{i,j}\sim H_{l,k}$ if $\{i,j\}\cap\{l,k\}\neq\emptyset$, and let $\approx$ be its extension to an equivalence relation (that is, $H_{i,j}\approx H_{l,k}$ if and only if there exists a chain of $\sim$ leading from $H_{i,j}$ to $H_{l,k}$).  Let $N_1,\ldots, N_v$ be equivalence classes of $\mathcal{A}_X$ arising from $\approx$, and for each $t$, let $L_t$ be the set of reflection across hyperplanes in $N_t$.
 We note that, $\langle L_1\rangle,\ldots,\langle L_s\rangle $ are subgroups of $A_n$ and $G_X=\langle L_1,\ldots,L_s\rangle$.  Furthermore $\langle L_1\rangle,\ldots,\langle L_s\rangle$ intersect trivially pairwise by the definition of $\approx$.  Thus $G_X=\langle L_1\rangle\times\dots\times \langle L_s\rangle$.

Moreover, if $H_{i,j}$, $H_{i,l}\in N_v$ for some $v\in\{1,\ldots,s\}$, then  $H_{j,l}\in N_v$.  Thus $\langle L_v\rangle\cong A_{n_v}$ where $n_v+1$ is the number of distinct variables appearing in elements of $L_v$.  Since $X$ is essential,  it follows that $n_v<n$.
\end{proof}

This allows one to inductively work out the containment problem for singular loci of the arrangements for $A_n$.

\begin{thm}
\label{thm:A}
For all $n\in\mathbb{N}$, $J(\mathcal{A}(A_n))^{(3)}\subseteq J(\mathcal{A}(A_n))^2$.
\end{thm}
\begin{proof}
We argue by induction on $n$.  For $n=1$, $J(\mathcal{A}(A_n))=(x_0-x_1)$ is a principal ideal and for $n=2$, $J(\mathcal{A}(A_n))=(x_0-x_1,x_1-x_2)$, a complete intersection.  Thus, in both cases, $J(\mathcal{A}(A_2))^{(3)}=J(\mathcal{A}(A_2))^3\subseteq J(\mathcal{A}(A_2))^2$.

If $n\geq 3$ then each associated prime of $J(A_n)^2$ corresponds to an essential flat by Lemma \ref{lem:locJ^2}.  Therefore, by Lemmas \ref{AnLoc} and \ref{lem:containmentproducts} and the inductive hypothesis, for each $P\in\Ass(J(\mathcal{A}(A_n))^2)$ we have the desired containment $J(\mathcal{A}(A_n))_P^{(3)}\subseteq J(\mathcal{A}(A_n))_P^2$.
By Theorem \ref{thm:localizationcriterion} the containment holds for $J(\mathcal{A}(A_n))$. 
\end{proof}

\subsubsection{Monomial and full monomial groups}
 In the following, we describe the structure of the singular loci of localizations of these arrangements and draw conclusions about symbolic power containments.
\begin{lem}
\label{lem:mon}
Let $\mathcal{A}=\mathcal{A}(G(m,m,n))$ or $\mathcal{A}=\mathcal{A}(G(m,1,n))$, and let $X$ be an essential flat of $\mathcal{A}$.  Then 
\[
\mathcal{A}_X=\mathcal{A}(G_1\times\dots\times G_v),
\] where each $G_i\in\{G(m,m,n_i),A_{n_i}\}$ for some $n_i<n$ in the former case and each $G_i\in\{G(m,1,n_i),A_{n_i}\}$ for some $n_i<n$ in the latter case
\end{lem}
\begin{proof}
We provide a complete argument for $\mathcal{A}=\mathcal{A}(G(m,m,n))$. The case $\mathcal{A}=\mathcal{A}(G(m,1,n))$ being analogous we omit the details.

Let $\xi$ be a primitive $m$th root of unity.  For each $i,j,s$, let $H_{i,j}^s=V(x_i-\xi^sx_j)$.  
Define a reflexive relation $\sim$ on $\mathcal{A}_X$ where $H_{i,j}^s\sim H_{l,k}^t$ if $\{i,j\}\cap\{l,k\}\neq\emptyset$, and let $\approx$ be its extension to an equivalence relation.  Let $N_1,\ldots, N_v$ be equivalence classes of $\mathcal{A}_X$ arising from $\approx$, and for each $t$, let $L_t$ be the set of reflection across hyperplanes in $N_t$.

Since $G_X=\langle L_1,\dots,L_v\rangle$ and since by construction $\langle L_1\rangle,\dots,\langle L_v\rangle$ intersect trivially, $G_X=\langle L_1\rangle\times\dots\times\langle L_2\rangle.$  
Let $N\in\{N_1,\ldots,N_v\}$, and let $L$ be the corresponding set of reflections. We show that either for each pair $i, j$ there is at most one $s$ such that $H_{i,j}^{s}\in N$ or for each pair $i, j$ so that $H_{i,j}^{s}\in N$ for some $s$, in fact $H_{i,j}^{s}\in N$ for all integers $1\leq s\leq m$. This suffices to conclude that $\langle L\rangle\cong G(m,m,n')$ with $n'<n$ in the first case and that  $\langle L\rangle\cong A(n')$ with $n'<n$ in the second case.

First note that if $H_{i,j}^{s_1}, H_{i,j}^{s_2}\in N$ for $s_1\neq s_2$ then $H_{i,j}^{s}\in N$ for $1\leq s \leq m$. For this it suffices to show that $H_{i,j}^{s}\in \A_X$ for $1\leq s \leq m$, which follows since $H_{i,j}^s\in (x_i, x_j)=(H_{i,j}^{s_1}, H_{i,j}^{s_2}) \subseteq I(X)$. Next assume towards  a contradiction that there exist indices $i, j, k$ so that  $H_{i,j}^{s_1}, H_{i,j}^{s_2}\in N$ for $s_1\neq s_2$ and $H_{j,k}^s\in N$ for a unique $s$. Notice that $I(X)\supseteq (H_{i,j}^s, H_{j,k}^{s_1}, H_{j,k}^{s_2})=(x_i, x_j, x_k)$ and therefore $X \subseteq V(H_{j,k}^s)$ and so  $H_{j,k}^s\in N$ for $1\leq s\leq m$, a contradiction. 
\end{proof}

Using this lemma one can prove containment properties for the singular loci of the families of arrangements $\mathcal{A}(G(2,2,n))$ and $\mathcal{A}(G(m,1,n))$.



\begin{thm}
\label{B,D,etc}
For $G\in\{G(2,2,n), G(m,1,n) \mid m\geq 2,n\geq 1\}$ there is a containment 
\[
J(\mathcal{A}(G))^{(3)}\subseteq J(\mathcal{A}(G))^2.
\]
\end{thm}
\begin{proof}
The proof is by induction on $n$, with $n\leq 3$ serving as the base case.

If $n<3$, then the singular locus of the arrangement is empty, so containment holds trivially.  For $n=3$,  $J(\mathcal{A}(G(2,2,n)))$ and $J(\mathcal{A}(G(m,1,n)))$ are generated by the $2\times 2$ minors of $3\times 2$ matrices  by Propositions \ref{prop:G(m,m,n)} and \ref{prop:G(m,1,n)}.  In both cases the ideal generated by the entries of this matrix is $(x_0,x_1,x_2)$, in particular it requires only three generators, hence by \cite[Theorem 5.1]{GHM}, the claimed containment holds.

If $n\geq 4$ then for each $P\in\Ass(J(\mathcal{A}(G))^2)$ there is a containment  $J(\A)_P^{(3)}\subseteq J(\A)_P^2$ obtained by applying Lemmas \ref{lem:containmentproducts}, \ref{lem:mon} as well as the inductive hypothesis for the factors given by (full) monomial groups and Theorem \ref{thm:A} for the factors given by symmetric groups. Therefore by Theorem \ref{thm:localizationcriterion}, we deduce $J(\A)^{(3)}\subseteq J(\A)^2$, completing the proof.
\end{proof}

\subsubsection{Conclusions} There are relatively few groups which can arise as the fixers of flats in complex reflection arrangements. For the irreducible complex reflection arrangements these are listed in Tables C.5-C.23 of \cite{OT}.  Using Theorem \ref{thm:localizationcriterion} and the results of subsections 5.2.1 and 5.2.2 one can determine exactly which sporadic irreducible complex reflection arrangements have singular loci whose defining ideals satisfy the containment $J(\A)^{(3)}\subseteq J(\A)^2$.
\begin{thm}
\label{Thm:Contain}
Let $\mathcal{A}(G)$ be a finite irreducible complex reflection group.  Then $J(\mathcal{A}(G))^{(3)}\subseteq J(\mathcal{A}(G))^2$ if $G\neq G_{24},G_{27},G_{29},G_{33},G_{34},$ or $G(m,m,n)$ with $m,n\geq 3$.
\end{thm}
\begin{proof}
The ideal $J(\A(G))$ is proper if and only if $\rank(G)\geq 3$, which implies that the containment is trivially satisfied for $G$ not in the family $G(m,p,n)$ with $n\geq 3$ or the sporadic groups $G_{23},\ldots G_{37}$. In the infinite family $G(m,p,n)$ the distinct arrangements correspond to the subfamilies $G(m,m,n)$ and $G(m,1,n)$. Theorem \ref{thm:A} shows the claimed containment holds for the groups $G(1,1,n)$ while Theorem \ref{B,D,etc} shows that containment holds for the groups $G(2,2,n)$ and $G(m,1,n)$.

Finally, among the sporadic group the claimed containment can be checked as follows. For the rank three groups $G_{23},G_{25}, G_{26}$ this follows by considering the ideal generated by the entries of the respective Hilbert-Burch matrix for $J(\A(G))$. Indeed, by Proposition \ref{prop:sporadictrue} the singular loci of these groups have a coefficient matrix of basic derivations as their Hilbert-Burch matrix. Moreover the Euler derivation $\sum_{i=0}^{2}x_i\frac{\partial}{\partial x_i}$ is the basic derivation of smallest degree, which means that one column of the Hilbert-Burch matrix is the vector of variables $\begin{bmatrix} x_0 & x_1&x_2 \end{bmatrix}^T$. Consequently the ideal generated by the entries of this Hilbert-Burch matrix is the homogeneous maximal ideal. Based on this, \cite[Theorem 5.1]{GHM} yields the claimed containment.

For the higher rank groups $G_{28}$, $G_{30}$, $G_{31}$, $G_{32}$, $G_{35}$, $G_{36}$, $G_{37}$ the containment can be checked by localization utilizing Theorem  \ref{thm:localizationcriterion}. One can look up the fixers of flats of codimension 3 in these arrangements in in Tables C.5-C.23 of \cite{OT} and verify that in each case these are among the rank three groups previously accounted for for which the containment holds. This implies that the claimed containment holds locally at each associated prime of  $J(\A(G))^2$, thus $J(\A)^{(3)}\subseteq J(\A)^2$ holds globally.
\end{proof}

In the next subsection we show that the statement of this theorem is sharp, that is, for the groups excluded in the statement the claimed containment does not hold.

\subsection{Noncontainment}
There are a number of reflection arrangements which are known in the literature to have singular loci whose defining ideals sa\-tis\-fy $J(\A)^{(3)}\not\subseteq J(\A)^2$. These include the arrangements determined by the monomial groups $G(m,m,3)$ for $m\geq3$. The singular locus of such an arrangement is termed a Fermat configuration of points in $\P^2$ in \cite{DST,HS} where the claimed non-containment is shown. Additionally the arrangements determined by the groups $G_{24}$ and $G_{27}$ have singular point configurations termed the Klein and the Wiman configurations respectively in \cite{BDH}, where the non-containment above is shown.  By Theorem \ref{thm:localizationcriterion} (2) we see that the singular loci of any reflection arrangements which can localize to one of these arrangements or equivalently contain $G(m,m,3), G_{24}$ or $G_{27}$ as fixers also must satisfy the same non-containment.  

\begin{thm}
\label{thm:noncontainment}
If $G$ is a complex reflection group, and $X$ is a flat of $\mathcal{A}=\mathcal{A}(G)$ such that the subgroup of $G$ fixing $X$ pointwise is isomorphic to $G(m,m,3)$ (for $m\geq 3$), $G_{24}$, or $G_{27}$, then $J(\A)^{(3)}\not\subseteq J(\A)^2$.
\end{thm}
\begin{proof}
Let $H\in\{G(m,m,3),G_{24},G_{27}\}$ such that $G_X\cong H$.  By \cite{DST,HS} or \cite{BDH} there is a non-containment $J(\A(H))^{(3)}\not\subseteq J(\A(H))^2$, which leads to the desired conclusion by Theorem \ref{thm:localizationcriterion} part (2). 
\end{proof}

%
%

For the irreducible complex reflection groups, the fixers of each of their flats are listed in  Tables C.5-C.23 of \cite{OT}, rendering our theorem above effective as follows.  

\begin{cor}  If $G$ is one of the irreducible complex reflection groups $G_{24},G_{27},G_{29},G_{33},G_{34}$ or $G(m,m,n)$ with $m,n\geq 3$, then $J(\mathcal{A}(G))^{(3)}\not\subseteq J(\mathcal{A}(G))^2$.
\end{cor}
\begin{proof}
This follows for the family $G(m,m,n)$ with $m,n\geq 3$ by induction on the rank, $n$, of the group. Indeed for $n=3$ the claim is shown in \cite{DST}. Since the fixer of the coordinate point $X=[0:0\cdots:0:1]$ in $G(m,m,n)$ is $G(m,m,n-1)$ the noncontainment for $G=G(m,m,n)$ follows from that for $G=G(m,m,n-1)$ by appealing to part (2) of Theorem \ref{thm:localizationcriterion}.  This result also follows from \cite{Sz}.

For the sporadic groups, it suffices to note that $G_{29}$ contains $G(4,4,3)$ as a fixer for a 0-dimensional flat, while $G_{33}$ contains $G(3,3,4)$ as a fixer for a 0-dimensional flat, and $G_{34}$ contains $G_{33}$ as a fixer for a 0-dimensional flat and appeal to part (2) of Theorem \ref{thm:localizationcriterion} once again.
\end{proof}

\subsection{Conclusion}

We are now able to prove our main Theorem \ref{Thm:A} from the Introduction, which we recall here.

\begin{thm}
\label{thm:Areprise}
Let $G$ be a finite complex reflection group with reflection arrangement $\A$.  Then $J(\mathcal{A})^{(3)}\subseteq J(\mathcal{A})^2$ if and only if no irreducible factor of $G$ is isomorphic to one of the following groups 
\[
G_{24},G_{27},G_{29},G_{33},G_{34}, \text{ or } G(m,m,n) \text{ with }m,n\geq 3.
\]
\end{thm}
\begin{proof}
By Lemma \ref{lem:containmentproducts}, if $G=G_1\times \cdots\times G_t$ is a product of irreducible groups then the containment $J(\mathcal{A})^{(3)}\subseteq J(\mathcal{A})^2$ holds if and only if the containments $J(\A_i)^{(3)}\subseteq J(\A_i)^2$ hold for all the arrangements $\A_i=\A(G_i), 1\leq i\leq t$. Now Theorems \ref{Thm:Contain} and \ref{thm:noncontainment} yield that the only irreducible complex reflection groups for which the containments being discussed do not hold are the ones listed in the claim. 
\end{proof}

We give a more algebraic version of the theorem above, based on our knowledge from section \ref{s:equations} of generators and relations for ideals of singular loci of arrangements.

\begin{cor}
\label{cor:linearsyz}
Let $G$ be a finite complex reflection group with reflection arrangement $\A$.  Then $J(\mathcal{A})^{(3)}\subseteq J(\mathcal{A})^2$ if and only if there is a linear relation (syzygy) among the minimal generators of $ J(\mathcal{A})$.
\end{cor}
\begin{proof}
For $G$ reducible, it can be deduced from Lemma \ref{lem:products2} that $J(\mathcal{A})$ admits a  linear syzygy if and only if the ideal defining the singular locus of some factor of $G$ admits a linear syzygy. This reduces the claim to the case of $G$ irreducible. Under this assumption, the claim can be verified on a case by case basis based on the above theorem and the presentations for the various ideals $J(\mathcal{A})$ given in section \ref{s:equations} . 
\end{proof}

\section{Stable Harbourne containments and open questions}
\label{s:questions}

Several questions and possible extensions of our work are currently open.
The first concerns the next instance in the series of containments of the type $J(\mathcal{A})^{(2r-1)}\subseteq J(\mathcal{A})^r$ proposed by Harbourne.

\begin{qst}
\label{q:5in3} 
Is the containment $J(\A)^{(5)}\subseteq J(\A)^{3}$ always satisfied for any reflection arrangement $\A$?
\end{qst}

An inductive proof along the lines given here for Theorem \ref{thm:Areprise} is possible, once the base case is resolved, but the base case in this situation would have to include irreducible arrangements in $\P^2$ as well as in $\P^3$. The answer to the above question for reflection arrangements in $\P^2$ is affirmative, by Proposition \ref{prop:2r-1inr} below, however for reflection arrangements in $\P^3$ the answer remains unknown.

\begin{qst}
\label{q:2r-1inr}
More generally, are the containments $J(\A)^{(2r-1)}\subseteq J(\A)^{r}$ always satisfied for any reflection arrangement $\A$ and any $r\geq 3$? Are these containments satisfied at least for $r\gg0$, that is do the singular loci of reflection arrangements satisfy the stable Harbourne conjecture?
\end{qst}

We are able to answer both Questions \ref{q:5in3} and \ref{q:2r-1inr}  in the affirmative for reflection arrangements of irreducible reflection groups of rank three and their products. For a homogeneous ideal $I$ we denote by $\alpha(I)$ the minimum degree of a nonzero element of $I$ and by $\reg(I)$ the Castelnuovo-Mumford regularity. Our methods rely on work of Herzog \cite{Herzog} made effective by our the description for the syzygies of the ideals $J(\A)$ in section \ref{s:equations} as well as on the papers \cite{Nguyen} and \cite{Nguyen2} which establish this for the groups $D_3$ and $B_3$ respectively.

\begin{prop}
\label{prop:2r-1inr}
Let $G$ be an irreducible reflection group of rank three with reflection arrangement $\A$. Then $J(\A)^{(2r-1)}\subseteq J(\A)^{r}$ for $r\geq 3$.
\end{prop}
\begin{proof} 
Let $J=J(\A)$. 
If $m\geq 3$ and $G\in \{G(m,m,3), G_{24}, G_{27}\}$ it is known from \cite[Theorem 2.1]{DHN}, \cite[Theorem 1.4]{BDH} that $\sup\{\frac{m}{r} \mid J^{(m)}\not\subseteq J^{r}\}=\frac{3}{2}$. Since $\frac{2r-1}{r}>\frac{3}{2}$, the desired conclusion follows. 

Otherwise, it follows from Propositions \ref{prop:symmetric}, \ref{prop:G(m,m,n)}, \ref{prop:G(m,1,n)} and \ref{prop:sporadictrue}  that the Hilbert-Burch matrix of $J$ has one of its columns given by the vector of variables in the ambient ring of $J$. Denote by $\fm$ the irrelevant ideal of $\P^2$. From \cite[Remark 3.2]{Herzog} it follows that $\fm^{2(r-1)}J^{(r)}\subseteq J^r$ for all $r\geq 1$, hence $\alpha(J^{(r)})\geq r\alpha(J)-2(r-1)$. On the other hand, by \cite[Theorem 2.5]{NS} the equality $\reg(J^r)=(r+1)\alpha(J)-2$ holds for each $r\geq 2$.

To show that $J^{(2r-1)}\subseteq J^{r}$ it suffices by \cite{BH} to show that $\alpha(J^{(2r-1)})\geq \reg(J^r)$. In turn by the estimates above it is sufficient to show 
\begin{eqnarray*}
& (2r-1)\alpha(J)-2(2r-2) \geq (r+1)\alpha(J)-2 \\
&(r-2)\alpha(J)\geq  4r-6\\
 & \alpha(J)\geq 4+\frac{2}{r-2}.
\end{eqnarray*}
If $G\not\in\{A_3, G(2,2,3), G(2,1,3), G(3,1,3), G_{25}\}$ it is easily verified from Propositions \ref{prop:symmetric}, \ref{prop:G(m,m,n)}, \ref{prop:G(m,1,n)} and \ref{prop:sporadictrue} that $\alpha(J)\geq 6$, settling the claim. Indeed, Propositions \ref{prop:symmetric} yields $\alpha(J)=3$ for $A_3$, Proposition \ref{prop:G(m,m,n)} yields $\alpha(J)=m+1$ for $G=G(m,m,n)$ and  Proposition \ref{prop:G(m,1,n)} yields $\alpha(J)=m+2$ for $G=G(m,1,3)$. For sporadic groups, Proposition \ref{prop:sporadictrue}  describes $J$ as the ideal of maximal minors for the Jacobian matrix $\Jac(f_1, f_2)$, where $f_1, f_2$ are the  basic invariants of lowest degree among a set of three basic invariants $f_1, f_2, f_3$ for $G$. This yields $\alpha(J)=(\deg(f_1)-1)+(\deg(f_2)-1)$.
Recall that the numbers $\deg(f_i)-1$ for $1\leq i\leq 3$ are called exponents of $G$. For sporadic groups the exponents are listed in table \eqref{table} in the proof of Proposition \ref{prop:sporadictrue} and the claim that $\alpha(J)\geq 6$ for $G\neq G_{25}$ can be directly verified. 

The remaining cases are summarize in the following table:
\begin{center}
\begin{tabular}{c|c|c}
Group & $\alpha(J)$ & description of $J$\\
\hline
$G(1,1,4)=A_3$ & 3 & $J=(xy(x-y),yz(y-z), xz(x-z))$\\
$G(2,2,3)=D_3$ & 3 & $J=(x(y^2-z^2),y(z^2-x^2),z(x^2-y^2))$\\\
$G(2,1,3)=B_3$ & 4 & $J=(xy(x^2-y^2),yz(y^2-z^2),xz(x^2-z^2))$\\
 $G(3,1,3)$ & 5 & $J=(xy(x^3-y^3),yz(y^3-z^3),xz(x^3-z^3))$\\
 $G_{25}$ & 5 &
\end{tabular}
\end{center}

The argument above settles the claim for $G\in\{G(3,1,3),G_{25}\}$ provided $r\geq 4$, since $\alpha(J)=5$ in these cases. For $G\in\{G(3,1,3),G_{25}\}$ and $r=3$ one can apply  \cite[Corollary 3.3]{Herzog} instead of  \cite[Remark 3.2]{Herzog} to strengthen the previous argument. Specifically,  \cite[Corollary 3.3]{Herzog} yields that $\fm^{r-1}J^{(r)}\subseteq J^r$ for $1\leq r\leq \alpha(J)$. In particular, for $r=5$, $\fm^4J^{(5)}\subseteq J^5$ holds, thus 
\[
\alpha(J^{(5)})\geq 5\alpha(J)-4\geq 4\alpha(J)-2=\reg(J^3),
\]
which leads to the desired containment $J^{(5)}\subseteq J^3$.

For $G=A_3$, we show that $\alpha(J^{(r)})\geq 3r$. In the reflection arrangement of $A_3$, the lines defined by $x,y,z$ each contain 3 singular points. Let $F\in J^{(r)}$ and let $k\geq 0$ be such that $F=(xyz)^kG$ and $G$ is not divisible by at least one of $x,y$ or $z$. Call the line defined by this form $L$. By Bezout's theorem, since $L$ intersects the zero set of $G$ with multiplicity $3(r-k)$ or $4(r-k)$ respectively, it follows that $\deg(G)\geq 3(r-k)$ or $\deg(G)\geq 4(r-k)$. Thus $\deg(F)\geq 3k+3(r-k)=3r$ or $\deg(F)\geq 3k+4(r-k)=4r-k$. Since $k\leq \deg(F)/3$ the latter possibility yields $4/3\deg(F) \geq \deg(F)+k\geq 4r$ which in turn simplifies to  $\deg(F)\geq 3r$. Since the inequality $\alpha(J^{(r)})\geq 3r \geq \alpha(J)(r+1)-2$ holds, the same argument as above establishes the claim. An independent proof for this case can be found in  \cite[Corollary1.6]{Nguyen2}. 

Finally, for the group $D_3$ the desired containment is shown in \cite[Corollary1.3]{Nguyen}. For the group $B_3$,  the desired containment is shown in   \cite[Proposition 5.2]{Nguyen2}. 

\end{proof}

\begin{rem}
 The second and third paragraphs of the proof above show more generally that if $I$ is a codimension two almost complete intersection ideal of $k[x,y,z]$ such that $\alpha(I)\geq 6$ and $I$ has a regular sequence of linear forms as a syzygy, then $I^{(2r-1)}\subseteq I^r$ for $r\geq 3$.
\end{rem}

Proposition \ref{prop:2r-1inr} allows us to prove the following general theorem. 

\begin{thm}
\label{thm:Breprise}
Let $G$ be a finite complex reflection group with irreducible factors of rank three and corresponding reflection arrangement $\A$.  Then for all integers $r\geq 3$, the containment $J(\mathcal{A})^{(2r-1)}\subseteq J(\mathcal{A})^r$ holds.
\end{thm}
\begin{proof}
By \ref{lem:containmentproducts}~(2) the claim is equivalent to showing the analogous containments hold for each of the irreducible factors of $G$, a fact established in Proposition \ref{prop:2r-1inr}.
\end{proof}

Three classes of arrangements with special properties have been singled out in the literature. These include  inductively free arrangements, first introduced by Terao in \cite{Ter80}, recursively free arrangements which were introduced by Ziegler in \cite{Zie87} and supersolvable arrangements. It is known that a reflection arrangement $\A(G)$ is recursively free if and only if $G$ does not admit an irreducible factor isomorphic to one of the exceptional reflection groups $G_{27}$, $G_{29}$, $G_{31}$, $G_{33}$ and $G_{34}$. On the other hand, a reflection arrangement $\A(G)$ is inductively free if and only if $G$ does not admit an irreducible factor isomorphic to a monomial group $G(m,m,n)$ with $m\geq 3$, $n\geq 3$, $G_{24}$, $G_{27}$, $G_{29}$, $G_{31}$, $G_{33}$ or $G_{34}$. Finally the arrangements $G(m,p,n)$ with $n\leq 2$ or $m\neq p$ are known to be supersolvable.

In view of these classifications, our results say that among  reflection arrangements all which are inductively free, recursively free or supersolvable satisfy the containment $J(\A)^{(3)}\subseteq J(\A)^2$. Moreover all reflection arrangements of rank three which are inductively free, recursively free or supersolvable satisfy the containments $J(\A)^{(2r-1)}\subseteq J(\A)^r$ for all $r\geq 2$. 
Motivated by this evidence, one can pose the following questions regarding the relationship between these properties of arrangements and the general containment problem.

\begin{qst}
Are the containments $J(\A)^{(2r-1)}\subseteq J(\A)^{r}$ always satisfied for any $r\geq 2$ and any hyperplane arrangement  that is supersolvable?
\end{qst}

\begin{qst}
Are the containments $J(\A)^{(2r-1)}\subseteq J(\A)^{r}$ always satisfied for any $r\geq 2$ and any hyperplane arrangement  that is inductively free?
\end{qst}

\begin{qst}
Are the containments $J(\A)^{(2r-1)}\subseteq J(\A)^{r}$ always satisfied for any $r\geq 2$ and any hyperplane arrangement  that is recursively  free?
\end{qst}

\section*{Acknowledgements}
We thank Elo\'isa Grifo, Jack Jeffries and Tomasz Szemberg for useful comments. We acknowledge Thai Nguyen for pointing out an error in a previous version of the proof of Proposition \ref{prop:2r-1inr}. The authors acknowledge the support of NSF grant DMS-1601024 and EpSCOR award OIA-1557417 throughout the completion of this work.

\addcontentsline{toc}{chapter}{Bibliography}

\end{document}